\documentclass[12pt]{amsart}
\usepackage{amsmath,amssymb,amsfonts,latexsym}
\usepackage{mathrsfs}
\usepackage[dvips]{geometry}
\usepackage{array}
\usepackage{epsf}
\usepackage{epsfig}
\newtheorem*{lemma}{Lemma}
\newtheorem*{prop}{Proposition}
\newtheorem*{thm}{Theorem}
\newtheorem*{cor}{Corollary}
\newcommand{\iso}{\overset{\sim}{\rightarrow}}

\begin{document}

\title[A Pentagonal Crystal]{A Pentagonal Crystal, the Golden Section,
alcove packing and aperiodic tilings.}

\author[Anthony Joseph]{Anthony Joseph\\}
\address {Donald Frey, Professional Chair\\
Department of Mathematics\\
The Weizmann Institute of Science\\
Rehovot, 76100, Israel}

\email {anthony.joseph@weizmann.ac.il}

\date{\today}
\maketitle
\medskip


\footnotetext[1]{Work supported in part by Minerva grant, no.
8596/1.}

 \

 \

\textbf{Abstract}

\

A Lie theoretic interpretation is given to a pattern with
five-fold symmetry occurring in aperiodic Penrose tiling based on
isosceles triangles with length ratios equal to the Golden
Section.   Specifically a $B(\infty)$ crystal based on that of
Kashiwara is constructed exhibiting this five-fold symmetry.  It
is shown that it can be represented as a Kashiwara $B(\infty)$
crystal in type $A_4$. Similar crystals with $(2n+1)$-fold
symmetry are represented as Kashiwara crystals in type $A_{2n}$.
The weight diagrams of the latter inspire higher aperiodic tiling.
In another approach alcove packing is seen to give aperiodic
tiling in type $A_4$. Finally $2m$-fold symmetry is related to
type $B_m$.

\

\section{Introduction}
\subsection{}\label{1.1}

This work arose as an attempt to explicitly describe and more
deeply understand the $B(\infty)$ crystal introduced by Kashiwara
\cite [Sections 0,4] {Ka1}.     Let us first recall the context in
which it is described.

\subsection{}\label{1.2}
 Let $C$ be a Cartan matrix in the sense used to define
Kac-Moody algebras.  More precisely $C$ is a square matrix of
finite (or even countable size) with diagonal entries equal to 2
and non-positive integer off-diagonal entries.  In the Kashiwara
theory one needs $C$ to be symmetrizable in order to introduce the
associated quantized enveloping algebra from which the (purely
combinatorial) properties of $B(\infty)$  are deduced. However
using the Littelmann path model \cite{Li1} it is possible (\ref
{1.6}) to weaken this to the requirement that the $ij^{th}$ entry
of $C$ be non-zero if and only if the $ji^{th}$ entry be non-zero.
This is of course exactly the condition under which the Kac-Moody
algebra $\mathfrak{g}$ associated to $C$ is defined \cite {Kac1}.

\subsection{}\label{1.3}
The $B(\infty)$  crystal is a purely combinatorial object which
can viewed as providing a basis (the crystal basis) of the algebra
of functions on the open Bruhat cell defined by $\mathfrak{g}.$
The latter is of course a polynomial algebra (on possibly infinite
many variables) though this is very far from obvious from the
combinatorial description of $B(\infty)$. Indeed one only knows
its formal character to have the expected product form by means
which are particularly roundabout, especially in the
non-symmetrizable case. Moreover the combinatorial complexity of
$B(\infty)$ is essential in that it leads in a simple manner to a
crystal basis for each highest weight integrable module.

\subsection{}\label{1.4}
The $B(\infty)$  crystal is specified entirely in terms of the
Cartan matrix $C$ by the following very simple procedure. First as
in Kac \cite[Chapter 2]{Kac1} one realizes $C$ through a vector
space ${\mathfrak h}$ (eventually the Cartan subalgebra of
${\mathfrak g}$), a set of simple coroots in ${\mathfrak h}$ and a
set $\Delta$ of simple roots in ${\mathfrak h}^*$ so that the
entries of C are given by evaluation of coroots on roots. Moreover
using these roots and coroots we may define the Weyl group $W$ in
the usual way.

\subsection{}\label{1.5}

To each simple root ${\alpha}$, Kashiwara \cite[Example 1.2.4]
{Ka2} introduced an "elementary crystal" $B_\alpha$ whose elements
can be viewed simply as non-positive multiples of $\alpha$, so
then $B_\alpha$ is identified with ${\mathbb N}$. (We have changed
Kashiwara's definition slightly (see \cite[12.3]{J1}). Now fix any
countable sequence $J$ of simple roots (indexed by the positive
integers) with the property that every simple root occurs
infinitely many times and take the subset $B_{J}$ of the
corresponding tensor product of the elementary crystals having
only finitely many non-zero entries. Of course $B_J$ has a
distinguished element, denoted $b_{\infty}$, in which all entries
are equal to zero. One views the elements of ${B_J}$ as forming
the vertices of a graph (the crystal graph). On $B_{J}$ one
defines a Kashiwara function $r$  with entries in $\mathbb Z$,
given by a very simple formula involving just the entries of the
Cartan matrix $C$. Its role is to describe the edges of the
crystal graph which are labelled by the simple roots. Indeed
inequalities between the values of $r$ on a given vertex $b$
decide the neighbours of $b$. Finally $B_{J}(\infty)$ is defined
to be the connected component of $B_{J}$ containing $b_{\infty}$ .

\subsection{}\label{1.6}

A deep and important result of Kashiwara is that as a graph
$B_{J}(\infty)$ is independent of $J$.   Kashiwara's result is
obtained via the quantized enveloping algebra using a
$q\rightarrow 0 $ limit. (Lusztig \cite{Lu1} has a different
version of this limit and the resulting combinatorics.) It
requires $C$ to be symmetrizable; but this condition can be
dropped through a purely combinatorial proof using the Littelmann
path model \cite[11.16, 15.11, 16.10]{J1}.

\subsection{}\label{1.7}

One can ask if it is possible to describe $B_{J}(\infty)$
explicitly as a subset of $B_{J}$.    Of course this should
involve the Cartan matrix which is in effect the only ingredient
in the determination of $B_{J}(\infty)$.   However this occurs in
an extremely complicated fashion and it is even rather difficult
to establish general properties of the embedding \cite {Ka3,N1}.
Nevertheless it was noted by Kashiwara \cite[Prop. 2.2.3] {Ka2}
that the rank $2$ case is manageable. Here we note that this
solution involves the Chebyshev polynomials with argument being
the square root of the product of the two off-diagonal entries of
$C$.  In truth these are not quite the Chebyshev polynomials as
customarily defined, however the difference will probably not
bother most readers.  We refer the fastidious to \ref {2.2}.

\subsection{}\label{1.8}

The square of the largest zero of the $n^{th}$ Chebyshev
polynomial is $< 4$ and tends to this value as n tends to
infinity. In particular the largest non-negative integer values
must be $0,1,2,3$ and these occur as the squares of the largest
zeros for just the second, third, fourth and sixth Chebyshev
polynomial. Moreover such a zero results in a cut-off in the
description of $B_{J}(\infty)$ which as a consequence lies in a
finite Cartesian product of elementary crystals.

\subsection{}\label{1.9}

One can ask if the squares of the largest zeros of the remaining
Chebyshev polynomials leads to a similar cut-off in the
description of $B_{J}(\infty)$.  The first interesting case is the
fifth Chebyshev polynomial whose largest zero is the Golden
Section $g$.  Of course since $g$ is not an integer or even
rational one needs to modify the definition of $B_{J}(\infty)$ for
the construction to make any sense.

\subsection{}\label{1.10}

In a similar vein one does not need the Cartan matrix to have
integer off-diagonal entries in order to define the Weyl group.
That the resulting group be finite (in rank two) similarly
involves the largest zeros of the Chebyshev polynomials.   In this
fashion the $n^{th}$ Chebyshev polynomial gives a Weyl group
isomorphic to the dihedral group of order $2n$.   Returning to the
case of $n = 5$ this leads to a "root system" having 5 positive
roots which matches with the expectation that $B_{J}(\infty)$
embeds in a five-fold tensor product.

\subsection{}\label{1.11}

The fact that the Golden Section $g$ is irrational and satisfies a
quadratic equation means that we may retain a purely integer
set-up in the definition of $B_J$ by adding collinear roots of
relative length $g$. This gives in all twenty non-zero roots.   A
link with mathematics of the ancient world is that the resulting
root system can be described as the projection of the vertices of
a dodecahedron onto the plane defined by one of its faces - see
Figure 1.

\subsection{}\label{1.12}

A further justification for introducing pairs of collinear roots
comes from the nature of  the Weyl group itself, which has two
generators and isomorphic to $\mathbb Z_5 \ltimes \mathbb Z_2$,
that is the dihedral group of order $10$. Because $g$ is not
rational but satisfies a quadratic equation there is a natural
decomposition of each of the two simple reflections into two
commuting involutions giving a larger group on four generators,
which we call the augmented Weyl group $W^a$, see \ref {3.9}. To
our surprise this larger group (which acts just ${\mathbb Z}$
linearly and not isometrically) leaves the enlarged root system
invariant. Having said this it is no surprise that this larger
group is just the permutation group on 5 elements and obtained as
the symmetry group of the dodecahedron via the projection
described in \ref{1.11}.

\subsection{}\label{1.13}

In view of the above rather pleasing geometric interpretation, it
became a seemingly worthwhile challenge to indeed construct and
describe explicitly a "pentagonal crystal", in the sense of the
Kashiwara $B_{J}(\infty)$, proving its independence on $J$ and
computing its formal character.

\subsection{}\label{1.14}

The above program was carried out, not without some difficulty.
Indeed the inequalities which describe $B_{J}(\infty)$ as a subset
of $B_J$ are significantly more complicated than a naive
interpretation of our previous inequalities from the rank $2$ case
would suggest.

\subsection{}\label{1.15}

The reader may spare himself the detailed verification of the
assertions alluded to in \ref{1.14} since it becomes apparent that
our pentagonal crystal is a realization of a Kashiwara
$B_{J}(\infty)$ crystal in type $A_4$ and $W^a$ the corresponding
Weyl group $W(A_4)$ in type $A_4$. Thus existence and independence
of $J$ may be proved by exhibiting this isomorphism. Nevertheless
our computation was not entirely in vain.  Indeed for some special
choices of J the description of $B_{J}(\infty)$ in type $A$ is
particularly simple \cite{N1}. However these are not the choices
required here and for them the resulting description is far more
complicated.

\subsection{}\label{1.16}

In view of \ref{1.15} it is natural to ask if the remaining
largest zeros of Chebyshev polynomials for $n \geqslant1$ lead to
crystals with a similar interpretation. Indeed the cases of the
third and fifth Chebyshev polynomials are just special cases of
the crystals obtained from the $(2n+1)^{th}$ Chebyshev polynomial.
It turns out that the $(2n+1)^{th}$ Chebyshev polynomial
factorizes into a pair of polynomials (related by replacing the
argument by its negative). These are irreducible over $\mathbb Q$,
if and only if $(2n+1)$ is prime. This leads to n collinear roots
replacing each positive (or negative) root leading to a total of
2n(2n+1) roots which just happens to be the number of roots in
type $A_{2n}$. In \ref {7.7} we exhibit the required isomorphism
with $B_{J}(\infty)$  for a particular choice of $J$. However one
should note that there is an important distinction with the
pentagonal case if $(2n+1)$ is not prime.  In particular we show
that the augmented Weyl group $W^a$ is just the Weyl group
$W(A_{2n})$ in type $A_{2n}$.  In Section 10, we consider the even
case for which $W\cong \mathbb Z_{2m} \ltimes \mathbb Z_2$.  We
show that then $W^a$ is isomorphic to the Weyl group $W(B_m)$ in
type $B_m$.  One would clearly like to extend this connection for
all finite reflection groups $W$, that is to say construct $W^a$
and show it to be isomorphic to the Weyl group of a root system.

\subsection{}\label{1.17}

We remark that although our pentagonal root system is just an
appropriate orthogonal projection of a dodecahedron, the latter
cannot be obtained from a similar orthogonal projection of the
root system of  type $A_4$. Yet the relation between the Coxeter
groups of type $A_4$ and the pentagonal system may be thought to
be an extension of the traditional one obtained by say embedding a
root system of type $G_2$ into one of type $B_3$ via a seven
dimensional representation of the Lie algebra of the former.

\subsection{}\label{1.18}

In Section 8 we consider Penrose aperiodic tiling based on the two
isosceles triangles (the Golden Pair \ref {8.7}) whose unequal
side lengths ratios is the Golden Section.  We view these
triangles as being obtained by triangularization of the regular
pentagon.  We show (Theorem \ref {8.14}) that the triangles
obtained by the regular $n$-gon lead to a higher aperiodic tiling,
though this is a totally elementary result having no Lie theory
content. In Section 9 we suggest that such tilings can be thought
of as a consequence of alcove packing in the Cartan subalgebras
whose associated Weyl group is the augmented Weyl group.  An
explicit construction is given in the pentagonal case, \ref
{9.10}, {9.11}. Aperiodicity (which we view as the possibility to
obtain arbitrary many tilings) corresponds to using different
sequences of reflections in the affine Weyl group. However for the
moment our construction does not give all possible tilings.

\section{Root systems}

       Throughout the base field will be assumed to be the real numbers
       $\mathbb R$.

\subsection{}\label{2.1}

Let ${\mathfrak h}$ be a vector space and $I : =\{1,2,\ldots,n\}$.
Define a root pair $(\pi^{\vee},\pi)$ to consist of a set
$\pi^{\vee}=\{\alpha_i^{\vee}\ |\ i \in I\}$ of linearly
independent elements (called simple coroots) of ${\mathfrak h}$
and a set $\pi= \{\alpha_ {i} \ |\ i\in  I\}$   of linearly
independent elements (called simple roots) of ${\mathfrak h^*}$
such that $\alpha_i^{\vee}(\alpha_i)  = 2$, for all $i \in I$. For
all $i \in I$, define  the simple reflection $s_i \in Aut\
{\mathfrak h^*}$ by
$$ s_i\lambda=\lambda-\alpha_i^{\vee}(\lambda)\alpha_i,$$
and let $W$ be the group they generate.  It will be assumed that
$\alpha_i^{\vee}(\alpha_j)  = 0$, if and only if
$\alpha_j^{\vee}(\alpha_i)  = 0$. The matrix with entries
$\alpha_i^{\vee}(\alpha_j)$  will be called the Cartan matrix. For
the moment we shall only assume that its off-diagonal entries are
non-positive reals.

\subsection{}\label{2.2}

Take $n = 2$ in \ref {2.1}.  Set $\alpha = \alpha_1,\beta
=\alpha_2, s_\alpha = s_1 , s_\beta = s_2$  .   Since we do not
mind introducing possibly superfluous square roots we shall
symmetrize the Cartan matrix so that its off-diagonal entries are
both equal to $-x$. Observe that $$s_\alpha s_{\beta}\alpha= (x^2
- 1)\alpha + x\beta , s_\alpha s_{\beta}\beta= - \beta-x\alpha.$$

         Thus if we define functions $R_n(x), S_n(x)$ by $$(s_\alpha
s_\beta)^n\alpha = R_n(x)\alpha + S_n(x)\beta,$$ we find that
$R_n,S_n$ are defined by the recurrence relations $$S_{n+1}=xR_n -
S_n, R_{n+1}=(x^2 - 1)R_n - xS_n = xS_{n+1}- R_n : n >  0,$$ with
the initial conditions $S_0 = 0, R_0   = 1$.

         These relations are exactly satisfied by setting
$$R_n = P_{2n}, S_n = P_{2n-1},$$
where the $P_n$ satisfy the recurrence relation
$$P_{n+1}= xP_n - P_{n-1},  \forall n \geq 0,$$
with the initial condition $P_{-1}= 0, P_0  = 1$. One may check
that $$(\sin\theta) P_n(2\cos\theta)=\sin(n+1)\theta, \forall n
\in \mathbb N.\eqno{(*)}$$

A few examples are given below
$$P_0=1,P_1=x,P_2=x^2-1,P_3=x(x^2-2),P_4=x^4-3x^2+1,P_5=x(x^2-3)(x^2-1),$$
$$P_6=x^6-5x^4+6x^2-1,P_7=x(x^2-2)(x^4-4x^2+2),P_8=(x^2-1)(x^6-6x^4+9x^2-1).$$
Set $\theta =\pi/n+1$. Then by $(*)$ the $2\cos t\theta :t \in
1,2,\ldots,n$, form the set of zeros of the degree $n$ polynomial
$P_n$.  Thus these zeros are pairwise distinct and real with
$x:=2\cos \pi/(n+1)$ being the largest. Moreover $x$ is just the
third length of the isosceles triangle with equal side lengths $1$
and equal angles $\theta$.  Finally $(\pi - \theta)$ is just the
angle between the vectors $\alpha$ and $\beta$ in Euclidean
two-space.
 We will refer to $P_n$ as the $(n+1)^{th}$ Chebyshev polynomial.
To be precise the "true" Chebyshev polynomials $P_n^c$ only
coincide with the $P_n$ for $n=0,1$. Otherwise they are defined by
the recurrence relation
$$P_{n+1}^c= 2xP_n^c - P_{n-1}^c,  \forall n \geq 0.$$ One may
check that
$$P_{n+1}^c(x) = P_{n+1}(2x) - xP_n(2x),\quad P_n^c(\cos \theta)=
\cos n\theta, \forall n \in \mathbb N.$$ A few examples are given
below
$$P_0^c=1,P_1=x,P_2^c=2x^2-1,P_3^c=4x^3-3x,P_4^c=8x^4-8x^2+1.$$

\subsection{}\label{2.3}

In the above situation the Weyl group $W = <s_\alpha,s_\beta>$ is
finite if and only if $(s_\alpha s_\beta )^n  = 1$, for some $n
\geqq 2$. Assume that n is the least integer with this property.
Observe that $n = 2$, exactly when $x = 0$.  In general $s_\alpha
s_\beta $ is a rotation by $2\pi/n$. Thus if $n$ is even, say $n =
2m$, then $(s_\alpha s_\beta )^m$ is a rotation by $\pi$ and so
acts by $-1$. This is satisfied by the vanishing of $S_m$ . In $n$
is odd, say $n = 2m+1$, then $$(s_\alpha s_\beta )^m\alpha =
(s_\beta s_\alpha )^{m+1}\alpha  = -s_\beta(s_\alpha s_\beta
)^m\alpha.$$ Consequently $(s_\alpha s_\beta )^m\alpha$ is a
multiple of $\beta$ and similarly $(s_\beta s_\alpha )^m\beta$ is
a multiple of $\alpha$. These conditions are satisfied by the
vanishing of $R_m$ .  In all cases $W \cong \mathbb Z_n \ltimes
\mathbb Z_2$ , that is the dihedral group of order $2n$.

\section{Crystals}
\subsection{}\label{3.1}

Adopt the conventions of \ref {2.1}.   A crystal $B$ is a
countable set whose elements are viewed as vertices of a graph
(the crystal graph, also denoted by $B$).   The edges of $B$ are
labelled by elements of $\alpha$ with the following two conditions
imposed.
\\

\textsl{\textit{1)    Removing all edges except one results in a
disjoint union of linear graphs.}}
\\

\textsl{\textit{2) There is a weight function $wt:B\rightarrow
{\mathfrak h^*}$ with the property that $wt\ b - wt\ b^\prime \in
\{\pm\alpha\}$, if $b,b^\prime$ are joined by an edge labelled by
$\alpha$.}}
\\

By 1), 2) we may define maps $e_\alpha,f_\alpha:B\rightarrow B\cup
\{0\}$ by $e_\alpha b^\prime = b$ if and only if $f_\alpha b =
b^\prime$, whenever $b,b^\prime$ are non-zero, with $e_\alpha$
(resp. $f_\alpha$ ) increasing (resp. decreasing) weight by
$\alpha$.

Any subset $B^\prime$ of $B$ inherits a crystal structure by
deleting all edges joining elements of $B^\prime$ to $B\setminus
B^\prime$ . We say that $B^\prime$ is a strict subcrystal of $B$
if no edges need be deleted, that is as a graph, $B^\prime$ is a
component of $B$.

Let $\mathscr E$ (resp.$\mathscr F$)  denote the monoid generated
by the $e_\alpha$ (resp. $f_\alpha$ ):$\alpha \in \pi$.  A crystal
is said to be of highest weight $\lambda \in {\mathfrak h^*}$ if
there exists $b \in B$ of weight $\lambda$ satisfying $\mathscr E
b = 0$ and $\mathscr F b = B$.
\subsection{}\label{3.2} A crucial component of crystal theory is tensor structure. As a
set the tensor product $B\otimes B^\prime$ of crystals
$B,B^\prime$ is simply the Cartesian product where the weight
function satisfies $wt(b\otimes b^\prime) = wt \ b + wt \
b^\prime$. In order to assign edges to the required graph
Kashiwara \cite[Definition 1.2.1] {Ka2} introduced auxillary
functions $\varepsilon_\alpha: B\rightarrow \mathbb Z
\cup\{-\infty\}:\alpha \in \pi$, with the property that
$\varepsilon_\alpha(e_\alpha b) = \varepsilon_\alpha (b) - 1$,
whenever $e_\alpha b \neq 0$.

      Let us pass immediately to a multiple tensor product
$B_n \otimes B_{n-1} \otimes\ldots \otimes B_1$.  On the
corresponding Cartesian product we shall define the edges through
the Kashiwara function given on the element $b_n \times b_{n-1}
\times\ldots \times b_1 : b_i \in B_i$ , by $$r^k_\alpha (b) =
\varepsilon_\alpha(b_k) -   \sum_{j=k+1}^n \alpha^{\vee}(wt \
b_j).$$

\subsection{}\label{3.3}

Let us assume for the moment that the Cartan matrix is classical,
namely has non-positive integer diagonal entries. Set
$$\varepsilon_\alpha(b) = \max\limits_k r_\alpha^k(b),  \ g_\alpha(b)=\max\limits_k\{\varepsilon_\alpha(b_k) = \varepsilon_\alpha(b)\},
 \ d_\alpha(b)=\min\limits_k \{\varepsilon_\alpha(b_k)
=\varepsilon_\alpha(b)\}.$$ The Kashiwara tensor product rule
\cite[Lemma 1.3.6]{Ka2} is given by
\\

\textsl{\textit{i) $e_\alpha b = b_n\times\ldots \times e_\alpha
b_t\times\ldots \times b_t$, where $t = g_\alpha(b)$,}}
\\

\textsl{\textit{ii) $f_\alpha b = b_n\times\ldots \times f_\alpha
b_t\times\ldots \times b_t$, where $t = d_\alpha(b)$.}}
\\

      One checks that this gives the Cartesian product a crystal structure.  It is manifestly
associative, but it is not commutative. When (i) (resp. (ii))
holds we say that $e_\alpha$ (resp. $f_\alpha$) enters $b$ at the
$t^{th}$ place.

\subsection{}\label{3.4}

When one begins to tamper with the entries of $C$ taking them to
be arbitrary reals, some of the required properties may fail,
particularly that noted in the last part of \ref {3.1}.  This is
already the case when the diagonal elements are replaced by
non-positive integers. For this case a cure has been given by
Jeong, Kang, Kashiwara and Shin \cite{JKKS}. Non-integer (real)
entries are also problematic. Indeed suppose that $f_\alpha b \neq
0$ and set $t = d_\alpha(b)$. This means that $r_\alpha^s(b) <
r_\alpha^t(b)$, for all $s < t$. On the other hand
$r_\alpha^t(f_\alpha b) = 1 + r_\alpha^t(b)$, whilst
$r_\alpha^s(f_\alpha b) = {\alpha^\vee}(\alpha) + r_\alpha^s(b) =
2 + r_\alpha^s(b)$, for $s < t$. Thus to obtain $e_\alpha f_\alpha
b = b$ , we require that $r_\alpha^s(b) < r_\alpha^t(b)$ to imply
$r_\alpha^s(b) \leqq r_\alpha^t(b) - 1$, which is true if the
Kashiwara functions take integer values, but may fail otherwise.

\subsection{}\label{3.5}

Non-integer values of the Kashiwara function are inevitable if $C$
has non-integer entries.   Our remedy is to assume that the
entries of $C$ take values in a ring which is free finitely
generated $\mathbb Z$ module $M = \mathbb Z g_1  +  \mathbb Z g_2
+ \ldots + \mathbb Z g_s$, with $g_1 = 1$. If the Cartan matrix
$C$ has real entries, then $M$ one might expect to take M to be a
subring of $\mathbb R$. However it is rather more convenient to
allow $M$ to have zero divisors. Given $m \in M$, let $m_i$ denote
its $i^{th}$ component (in which we always omit the multiple of
$g_i$). Define
$$P =\{\lambda \in {\mathfrak h^*}| \alpha^\vee(\lambda) \in M,
\forall \alpha \in \pi\},P^+=\{\lambda \in
P|\alpha^\vee(\lambda)_i\geq0,\forall \alpha \in \pi,\forall i \in
\{1,2,\ldots,s\}\}.$$

      Assume that $wt$ takes values in $P$.  Further assume that
the $\varepsilon_\alpha$ take values in $M$. Consequently the
Kashiwara function will also take values in $M$.  Let ${\bf g}_i$
denote the element in M with $g_i$ in the $i^{th}$ entry and zeros
elsewhere.

      Following the above we extend the notion of a crystal in the
following obvious fashion.   Define for all $\alpha \in \pi$ and
$i \in \{1,2,\ldots,s\}$ maps $e_{\alpha,i}, f_{\alpha,i}: B
\rightarrow B \cup \{0\}$ satisfying $e_{\alpha,i}b^\prime = b$,
if and only if $f_{\alpha,i}b =  b^\prime$ whenever $b, b^\prime$
are non-zero with $e_{\alpha,i}$ (resp. $f_{\alpha,i}$) increasing
(resp. decreasing) weight by ${\bf g}_i\alpha$ and decreasing
$\varepsilon_\alpha$(resp. increasing) by ${\bf g}_i$.

      Continue to define the Kashiwara functions as in \ref {3.2}. Its
component in the $i^{th}$ factor will be an integer and as in \ref
{3.3} provides the rule for the insertion of the $e_{\alpha,i},
f_{\alpha,i}$, in a multiple product.

Notice that our algorithm is being applied to each simple root at
a time.  Thus we can allow ourselves the flexibility of using
different $\mathbb Z$ bases for $M$.   In some cases it is even
convenient to allow $M$ itself to depend on the simple root in
question (see Section 10).

\subsection{}\label{3.6}

Return to the case of a classical Cartan matrix.   Here Kashiwara
\cite[Example 1.2.6] {Ka2} introduced "elementary" crystals.  We
follow \cite [12.3]{J1} in modifying slightly their definition
which is given below.
$$B_\alpha =
\{b_\alpha(m): m \in \mathbb N : \forall \alpha \in \pi\}.$$

     Their crystal structure is given by
     $$wt\ b_\alpha(m) = -m\alpha, \varepsilon_\alpha(b_\alpha(m))=m,$$

and
     $$e_\alpha b_\alpha(m)=\left\{\begin{array}{ll}0& :\ m=0,\\
b_\alpha(m-1)& :\   m > 0,\\
\end{array}\right.$$

$$\varepsilon_\beta(b_\alpha(m)) = -{\infty},\ e_\beta(b_\alpha(m)) = f_\beta(b_\alpha(m)) = 0,\ for\ \beta \neq\alpha.$$
\
Since $B_\alpha$ has linear growth, we refer to it as a one
     dimensional crystal.
     Let $J$ be a countable sequence $\alpha_{i_m},\alpha_{i_{m-1}}, \ldots,\alpha_{i_1}$, of simple roots
in which each element of $\pi$ occurs infinitely many times. Then
we may form the countable Cartesian product $$\ldots\times
B_{i_n}\times B_{i_{n-1}}\times\ldots \times B_{i_1},$$ where
$B_{i{_m}}$ denotes $B_\alpha$, when $\alpha = \alpha_i{_m}$. We
note an element b of this product simply by the sequence
$\{\ldots, m_n, m_{n-1},\ldots,m_1\}$ of its entries.  Now let the
$B_J$ denote the subset in which all but finitely many $m_i$ are
equal to zero. Then the expression for the Kashiwara function is a
finite sum and through it we obtain a crystal structure on $B_J$.

      Note that $B_J$ has a distinguished element $b_{\infty}$, in which all
the $m_i$ are equal to zero.   It may also be characterized by the
property that $\varepsilon_\alpha(b_{\infty})=0, \forall \alpha
\in\pi$. Set $B_J({\infty})=\mathscr F b_{\infty}$.

      From Kashiwara's work \cite {Ka1,Ka2} one may immediately deduce
some quite remarkable facts about $B_J({\infty})$ given that $C$
is symmetrizable.   These, noted in \ref {3.7} below, can be
extended (\ref {1.6}) to all classical $C$ through the Littelmann
path model. We remark that if $J$ is non-redundant in the sense
that every entry is non-zero for some $b \in B_J({\infty})$, then
$\ldots s_{i_n}s_{i_{n-1}},\ldots,s_{i_1}$, can be taken to be
reduced decompositions of a sequence of elements of $W$.

      N.B.  Here there is a small but annoying subtlety (see
      \cite[3.13]{J3}) which one may only notice when one gets down to nitty
      gritty calculations.  In the formulation of Kashiwara a factor
      of $B(\infty)$ (of which only $b_{\infty}$ is needed) is carried to the left.  This ensures that
      the values of the $\varepsilon_\alpha$ stay non-negative on
      the elements of $B_J({\infty})$, which in turn is needed for
      the properties described in \ref {3.7} to hold.
      Alternatively one may add "dummy" factors causing some
      redundancy. Thus for any $s_\alpha:\alpha \in \pi$ which occurs only finitely
      many times in the above sequence one adds to $B_J$
      one additional factor of $B_\alpha$ at any place to their left. On
      this the corresponding entry stays zero for all $b \in B_J({\infty})$.

\subsection{}\label{3.7}

Fix $\alpha \in \pi$.  After Kashiwara a crystal B is called
$\alpha$-upper normal if $$\varepsilon_\alpha(b) = max\{n |
e_\alpha^n  b = 0\},\quad \forall b \in B.$$

A crystal is called upper normal if it is $\alpha$-upper normal
for all $\alpha \in \pi$. For example $B_\alpha$ is $\alpha$-upper
normal; but not upper normal. Though we barely need this we remark
that a crystal $B$ is said to be lower normal if for all $\alpha
\in B$ one has
$$\varphi_\alpha(b) = max\{n | f_\alpha^n  b = 0\},\quad
\forall b \in B,$$where $\varphi_\alpha(b) = \varepsilon_\alpha(b)
+\alpha^\vee(wt b)$. A crystal is said to be normal if it is both
upper and lower normal.

     Let $B$ be a subset of $B_J$ viewed as
a subgraph by just retaining all edges joining elements of $B$.
Because the weights of B must lie in $-\mathbb N\pi$, the subset
$B^\mathscr E : =\{b \in B|e_\alpha b = 0,\forall \alpha \in \pi
\}$ is non-empty. Upper normality of $B$ implies that $B^\mathscr
E = \{b_\infty \}$  by the remark in \ref {3.6}.

      The first remarkable result of Kashiwara is that $B_J({\infty})$ is upper normal.
By the remarks above this immediately implies that $B_J({\infty})$
is a strict subcrystal of $B_J$.

      The second remarkable result of Kashiwara is that as a crystal (or
equivalently as a graph) $B_J({\infty})$ is independent of $J$. We
denote it by $B({\infty})$.

      The third remarkable result of Kashiwara is that
$$ch B({\infty}) :  = \sum_{b \in B({\infty})}e^{wt\ b}= \prod_{{\alpha}\in \Delta^+}(1-e^\alpha)^{-m_\alpha},$$
where $\Delta^+$ denotes the set of positive roots of the
corresponding Kac-Moody algebra and $m_\alpha$ the multiplicity of
$\alpha$ root space. We remark that in order to extend this result
to the non-symmetrizable case we need the Littelmann character
formula for $B({\infty})$, which expresses the latter as an
alternating sum over $W$, together with the corresponding Weyl
denominator formula. The validity of the latter for the
non-symnmetrizable case was established independently by Mathieu
\cite {M} and Kumar \cite {Ku}.

\subsection{}\label{3.8}

Suppose now that $C$ admits off-diagonal entries with values in
$M$ as defined in \ref {3.5}.   Then we modify the elementary
crystals to take account of the additional elements $e_{\alpha,i},
f_{\alpha,i}$, introduced there. For this we set ${\mathbf g}
=(g_1,g_2,\ldots,g_s)$ and let ${\mathbf m} =
(m^1,m^2,\ldots,m^s)$ denote an arbitrary element of ${\mathbb
N}^s$ setting $$\mathbf g.\mathbf m = \sum_{i=1}^s g_im^i.$$

        Now define $$B_\alpha = \{b_\alpha(\mathbf m): \mathbf m \in \mathbb N^s \},$$

given the obvious crystal structure extending \ref {3.6}.  In
particular $$wt\ b_\alpha(\mathbf m) = -(\mathbf g.\mathbf
m)\alpha, \\\ \varepsilon_\alpha(b_\alpha(\mathbf m))=(\mathbf
g.\mathbf m),$$

$$e_\alpha b_\alpha(\mathbf m)=\left\{\begin{array}{ll}0& :\ m^i=0,\\
b_\alpha({\mathbf m}-{\mathbf g}_i)& :\   m^i > 0.\\
\end{array}\right.$$

      Notice for example that the $f_{\alpha,i}:i=1,2,\ldots,s$, commute on $B_\alpha$.
Via the insertion rules interpreted through \ref {3.5} they also
commute on $B_J$ and hence on $B_J(\infty)$.
      Note that $B_\alpha$ is now an s-fold tensor product of
      one-dimensional crystals, so we call it an $s$-dimensional
      crystal.
      It is not a priori obvious that this new $B_J(\infty)$ will still retain all
the good properties described in 3.7.

\subsection{}\label{3.9}

For each $\alpha \in \pi$, define $\mathbb Z$ linear maps on $P$
by $$s_{\alpha,i}\lambda=\lambda -
\alpha^\vee(\lambda)_ig_i\alpha:i=1,2,\ldots,s,$$ where as before
the $i^{th}$ subscript denotes projection onto the $i^{th}$
component of $M$, which we recall has values in $\mathbb Z$. One
checks that
$$s_{\alpha,i}s_{\alpha,j}\lambda = \lambda-\alpha^\vee(\lambda)_jg_j\alpha -
 \alpha^\vee(\lambda)_ig_i\alpha+2\delta_{i,j}g_i\alpha^\vee(\lambda)_j,$$
where $\delta$ is the Kronecker delta.   Thus the $s_{\alpha,i}$
commute pairwise, are involutions and their product is $s_\alpha$.
This commutation property parallels that noted in \ref {3.8}. Let
$W^a$ denote the group generated by the
$<s_{\alpha,i}:i=1,2,\ldots,s>$. We call it the augmented Weyl
group $W^a$.  A priori its structure depends on the choice of
generators for $M$. One can ask if choices can be made so that
$W^a$ is a Coxeter group and finite whenever $W$ is finite. We
investigate these questions in the rank 2 case, that is when
$|\pi| = 2$.  Apart from the pentagonal case (see Section 5) where
we discovered inadvertently that $W^a$ is the Weyl group for
$\mathfrak {sl}(5)$, our reasoning is somewhat a posteriori.

\section{Rank Two}
\subsection{}\label{4.1}

Fix $J$ as in \ref {3.6} and recall the definition of
$B_J(\infty)$. As noted in \ref {3.7} it may be viewed as a
presentation of the Kashiwara crystal $B(\infty)$ which is in turn
a combinatorial manifestation of a Verma module (or more properly
its $\mathscr O$ dual). As a set $B_J(\infty)$ is completely
determined by its specification as a subset of countably many
copies of $\mathbb N$. We would like to determine this subset
explicitly. This seems to be rather difficult; yet Kashiwara
\cite[Prop. 2.2.3] {Ka2} found an elegant solution which we derive
below by a different method which is applicable to the pentagonal
crystal.

\subsection{}\label{4.2}

Take $\pi= \{\alpha_1,\alpha_2 \}$.  Then the off-diagonal
elements of $C$ are $a: = -\alpha^\vee_1(\alpha_2),a':=
-\alpha^\vee_2(\alpha_1)$, which are integers $\geq 0$. Set $y =
aa'$.   Since we may now care about preserving integrality we do
not yet symmetrize $C$ as in \ref {2.1}.

      There are just two possible non-redundant choices for $J =
      \{\ldots.j_n,j_{n-1},\ldots,j_1\}$. The first is given by
       $$j_k=\left\{\begin{array}{ll}1& :\ k \ odd,\\
2& :\ k \ even.\\
\end{array}\right.$$

The second by interchange of $1,2$.   We consider just the first.

           Suppose y = 0.  Then one easily checks that $B_J(\infty)=\mathbb N^2$.  Assume $y > 0$.   Set
           $$c_k=\left\{\begin{array}{ll}a& :\ k \ odd,\\
a'& :\ k \ even.\\
\end{array}\right.$$

           Define the rational functions $T_n: n \geq 3$ by $T_3(y)= 1$ and
$$T_{n+1}(y) = 1- \frac{1}{yT_n(y)} :\forall n \geq 3.$$

           Recalling the conventions of \ref {3.6} let $\mathbf m
           =\{\ldots,m_n,m_{n-1},\ldots,m_1\}$
denote an element of $B_J$.

 \begin{lemma}  One has $\mathbf m \in B_J(\infty)$ if and only if $m_n \leq c_nT_n(y)m_{n-1}, \forall n \geq 3.$
   \end{lemma}

           \begin{proof}   It suffices to show that the subset of
           $B_J(\infty)$ defined by the right hand side of the lemma is $\mathscr E$ stable,
$\mathscr F$ stable and that its only element annihilated by
$\mathscr E$ is $b_{\infty}$.

          For $n$ odd $\geq 1$ one has
$$r_{\alpha_1}^{n+2}(\mathbf m)-r_{\alpha_1}^n(\mathbf m) = am_{n+1}- m_{n+2} -
m_n.$$

          Recall that if $e_{\alpha_1}$(resp $f_{\alpha_1})$ enters $\mathbf m$ at the $n^{th}$ place
then the above expression must be $< 0$ (resp. $\leq 0$).

           Assume the inequality of the lemma for $n$ replaced by $n+2$,
           whenever $n \geq 1$.   Suppose $e_{\alpha_1}$ enters at the $n^{th}$ place.  If $m_n = 0$, then
$e_{\alpha_1}\mathbf m = 0$.  Otherwise $m_n$ is reduced by one.
Yet $$m_n - 1 \geq am_{n+1} - m_{n+2} \geq a(1 -
T_{n+2}(y))m_{n+1} = \frac{m_{n+1}}{a'T_{n+1}(y)}.$$

It follows that the right hand side of the lemma is satisfied by
$e_{\alpha_1}\mathbf m$.   A similar result holds for n even.
This establishes stability under $\mathscr E$.   A very similar
argument establishes stability under $\mathscr F$.

        Finally assume $e_{\alpha_1}\mathbf m = e_{\alpha_2}\mathbf m = 0$.  Suppose $\mathbf m \neq b_{\infty}$
and let $m_n$ be the last non-vanishing entry of $\mathbf m$. One
easily checks that $n \geq 3$.  Then the inequalities of the lemma
force $m_i > 0$, for all $i < n$.   Assume n even.  Then the above
vanishing implies that $e_{\alpha_2}$ goes in at the $(n+2)^{nd}$
place which through the above expression for the differences of
Kashiwara functions gives the contradiction $-m_n \geq 0$. The
case of n odd is similar.
\end{proof}
\subsection{}\label{4.3}
One may easily compute the rational functions $T_n(y)$ for small
n. One obtains

  $$T_3 =1,\ T_4=\frac{y-1}{y}, T_5=\frac{y-2}{y-1},\ T_6
  =\frac{y^2-3y+1}{y(y-2)},\
T_7=\frac{(y-3)(y-1)}{y^2-3y+1},$$
  $$T_8 =\frac{y^3-5y^2+6y-1}{y(y-3)(y-1)},\quad T_9 =\frac{(y-2)(y^2-4y+2)}{y^3-5y^2+6y-1}.$$                                      .

           One may easily check the

               \begin{lemma}  If $y$ is real and $ \geq4$, then $1 > T_n(y) > \frac{1}{2},\forall n > 3$.
   \end{lemma}

\subsection{}\label{4.4}
The $T_n(y)$ are related to the Chebyshev polynomials $P_m(x)$
defined in \ref {2.2} by the formula
$$T_n(x^2)=\frac{P_{n-2}(x)}{xP_{n-3}(x)}, \forall n > 3. \eqno{(*)}$$
\subsection{}\label{4.5}
In applying $(*)$ to Lemma \ref{4.2}, we recall that the case $x =
0$ has been excluded.   Moreover if $P_{n-3}(x) = 0$, for a chosen
value of $x$, then by Lemma \ref{4.2} one has $m_{n-1} = 0$ and so
$m_n= 0$ also. Apart from the case $x = 0$ corresponding to $\pi$
of type $A_1 \times A_1$, the remaining cases when $B(\infty)$
lies in a finite tensor product are when $P_{n-2}(x)$ has a zero
for some positive integer value of $y = x^2$.  By \ref {4.3} we
must have $y < 4$. Thus there are three such possible values of
$y$, namely $1, 2, 3$ and these respectively are zeros of
$T_4,T_5,T_7$. They correspond to types $A_2,B_2,G_2$.
\subsection{}\label{4.6}
One can ask if real positive non-integer zeros of  $P_n(x)$ can
also lead to $B_J(\infty)$ being embedded in a finite tensor
product. The first interesting case is $P_4(x)$, namely the fifth
Chebyshev polynomial in our conventions. One has $$P_4(x) = x^4 -
3x^2 + 1  = (x^2- x - 1)(x^2 + x - 1).$$

      The positive roots of this polynomial are the Golden Section $g$ and
its inverse.   One may remark that $y = g^2 = 2.618\ldots$, and
lies between $2$ and $3$, the latter respectively corresponding to
types $B_2$ and $G_2$.
      From the point of view of the Weyl group
(see \ref {2.3}) which is isomorphic to  $\mathbb Z_5\ltimes
\mathbb Z_2$ in this case, we might expect $B_J(\infty)$ to be a
strict subcrystal of a five-fold tensor product of elementary
crystals. As pointed out in \ref {3.4} some modifications are
necessary since the Kashiwara functions will not be
integer-valued.  Now $M : = \mathbb Z[g]$ is a free rank 2 module.
Thus each elementary crystal should itself be a tensor product of
one-dimensional crystals and so ultimately $B_J(\infty)$  should
lie in a ten-fold tensor product of one-dimensional crystals. From
this one may anticipate that the underlying root system should
have ten positive roots coming in five collinear pairs of relative
length $g$.  To realize this we shall make two choices which
ultimately may affect the result, namely we choose $g_1 = 1, g_2 =
g$ in \ref {3.5} and the Cartan matrix to have off-diagonal
entries both equal to $-g$. Remarkably the resulting root system
is stable under the augmented Weyl group (\ref {3.9}) itself
isomorphic to $S_5$, the permutation group on five symbols.

\section{The Pentagonal Crystal}
\subsection{}\label{5.1}

Recall the notation and conventions of \ref {2.2} and take the
Cartan matrix to have off-diagonal entries equal to $-g$, where
$g$ is the Golden Section.  Choose $J$ as in \ref {4.2} with
$B_\alpha$, $B_\beta$ the elementary crystals defined in \ref
{3.8} with $g_1=1,g_2=g$. We write $e_{\alpha,1}= e_\alpha,\
e_{\alpha,2} =e_{g\alpha}$, and so on. The entries in the $i^{th}$
factor of $B_J$ will be denoted by $\mathbf m_j=(m_j,n_j)$.

\subsection{}\label{5.2}

The aim of this section is to give an explicit description of
$B_J(\infty)$ as a subset of $B_J$ in a manner analogous to \ref
{4.2}. In particular we show that $B_J(\infty)$ lies in a
five-fold tensor product of elementary crystals.  We show that as
a crystal it is independent of the two possible non-redundant
choices of $J$. We denote the resulting crystal by $B(\infty)$. We
show that the formal character of $B(\infty)$ has ten factors each
corresponding to the ten positive roots alluded to in \ref {4.6}.

\subsection{}\label{5.3}

We may regard $B_J$ as a repeated tensor product of the crystal
$B_\beta \otimes B_{g\beta} \otimes B_\alpha \otimes B_{g\alpha}$
defined via \ref {3.6}, {3.7}. Set $\mathbf m = (m,n)$. A given
element $b\in B_J$ is given by a finite sequence $({\mathbf
m}_j,\ldots,{\mathbf m}_1): =
(m_j,n_j,m_{j-1},n_{j-1},\ldots,m_1,n_1)$ of non-negative
integers, though $j$ may be arbitrarily large.

Recall \ref {3.5} that the Kashiwara functions take values in
${\mathbb Z}\oplus{\mathbb Z}g$. Their components in the first
(resp. second) factor will be denoted by $r_\alpha^j,r_\beta^j$
(resp.$r_{g \alpha}^j,r_{g \beta}^j$). These integers determine
the places in $b$ at which the crystal operators enter via the
rules $i), ii)$ of \ref {3.3}.  As in \ref {4.2}, it is only
certain differences that matter.

Take $b =({\mathbf m}_j,\ldots,{\mathbf m}_1)$ defined as above.
For $j$ odd $\geq 1$, we obtain $r_\alpha^{j+2}(b)- r_\alpha^j(b)$
(resp. $r_{g\alpha}^{j+2}(b) - r_{g\alpha}^j(b))$ as the
coefficient of 1 (resp. $g$) in the expression $$g(\mathbf
m_{j+1}.\mathbf g)-(\mathbf m_{j+2}.\mathbf g) - (\mathbf
m_j.\mathbf g) = g(m_{j+1} + gn_{j+1}) - (m_{j+2}+ gn_{j+2}) -
(m_j + gn_j ).$$

Then the identity $g^2 = g + 1$ gives
$$r_\alpha^{j+1}(b)-r_\alpha^j(b) = n_{j+1} - m_{j+2} - m_j, \  r_{g\alpha}^{j+2}(b) -
r_{g\alpha}^j(b) = m_{j+1} + n_{j+1} - n_{j+2} -  n_j.$$

When $j$ is even $\geq 2$, the above formulae still hold but with
$\alpha$ replaced by $\beta$.

\subsection{}\label{5.4}

At first one might expect the exact analogue of \ref {4.2} to hold
with the inequality $$(\mathbf m_j.\mathbf g)\leq
gT_j(g^2)(\mathbf m_{j-1}.\mathbf g),\eqno{(*)}$$ suitably
interpreted.   Since $T_6(g^2) = 0$ this would give $B_J(\infty)$
to lie in a five-fold tensor product of two dimensional crystals.
The true solution is more complex.  It is given by the following

\begin{prop}   One has $({\mathbf m}_j,\ldots,{\mathbf m}_1) \in B_J(\infty)$ if and only if

$(i)  \  m_k, n_k = 0 : k \geq 6.$

$(ii) \   m_3 = n_2-u,n_3 = m_2 + n_2 - v :  u,v \geq 0.$

$(iii) \  m_4 = n_2 - v - s, n_4 = m_2 + n_2 - u - v - t :$

       \quad \quad $u + t \geq 0, v + s \geq 0, v + t \geq 0 , s + t + v \geq
        0.$

$(iv)  \ m_5 = m_2 - v - t - a, n_5 = n_2 - u - v - s - t - a':$

       \quad \quad $v + t + a \geq 0, u + t + a' \geq 0, s + v + t + a \geq 0,$

       \quad \quad $s + v + t + a' \geq 0, s + t + v + a + a' \geq
       0.$
                 \end{prop}

Remark 1  It is implicit that $m_k,n_k\geq 0$, for all $k \in
\mathbb N^+$, and this gives some additional inequalities.

Remark 2. Recall that $T_3(g^2) = 1$.  Interpret $m + gn \leq m' +
gn'$ to mean $m \leq m', n \leq n'$.  Then $(ii)$ can be
interpreted as $(*)$ for k = 3.   However $(iii)$ and $(iv)$
cannot be similarly interpreted.

Remark 3. Set $b = (\mathbf m,\mathbf n)$.    Notice that

$$r_\beta^6(b)-r_\beta^4(b) = n_5  - m_4  = - (u + t + a')\leq 0,$$
$$r_{g\beta}^6(b) - r_{g\beta}^4(b) = m_5  + n_5  - n_4  = - (v + s + t + a + a') \leq0.$$

These imply $(i)$ above.

Remark 4.  In $5.9$ we describe an algorithm giving these
inequalities.

\subsection{}\label{5.5}

The proof of Proposition \ref {5.4} follows exactly the same path
as the proof of Lemma \ref {4.2}, as do also the remaining
assertions in \ref {5.2}. However verification of the details
should only be attempted by a certified masochist. We give a few
vignettes from the proof. These illustrate the ubiquitous nature
of the inequalities.

\subsection{}\label{5.6}

Define $B_J(\infty)$ by the inequalities in the right hand side of
Proposition \ref{5.4}.  We shall first illustrate stability under
$\mathscr E$ in the more tricky cases.

Take $b =({\mathbf m}_j,\ldots,{\mathbf m}_1)$ as before and
suppose that $e_\alpha$ enters $b$ at the third place.  This
implies in particular that
$$0 > r_\alpha^5(b)- r_\alpha^3(b)= n_4 - m_5 - m_3 = a.$$

Moreover if $e_\alpha b \geq 0$, then $m_3 \geq 1$ and this
insertion decreases $m_3$ by 1.  The latter can be achieved
without changing the remaining entries in $b =({\mathbf
m}_j,\ldots,{\mathbf m}_1)$ by increasing $u$ and $a$ by 1 and
decreasing $t$ by 1.  Then the only inequalities for the new
variables that might fail are those involving $t$ but neither $a$
nor $u$. One is thus reduced to the expressions $v + t, v + t + s
+ a', s + t + v$.  Yet in the old variables all these expressions
are $> 0$ in view of the fact that $a + v + t, a + v + t + s + a',
a + s + t + v$ are all non-negative, whilst as we have seen above
$a < 0$.

Suppose as a second example that $e_\beta$ enters $b$ in the
fourth place. This implies in particular that
$$0 > r_\beta^6(b) - r_\beta^4(b)= n_5- m_6 - m_4  =  -(u + t +
a').$$

Moreover if $e_\beta b\neq  0$, then $m_4 \geq 1$ and the
insertion of $e_\beta$ decreases $m_4$ by 1.  The latter can be
achieved without changing the remaining entries in $b =({\mathbf
m}_j,\ldots,{\mathbf m}_1)$ by increasing $s$ by 1 and decreasing
$a'$ by 1. Then the only inequalities for the new variables that
might fail are those involving $a'$ but not $s$. This is just $u +
t + a'$ which is strictly positive by the above.

\subsection{}\label{5.7}

Retain the above conventions.  We illustrate stability under
$\mathscr F$ in some of the more tricky situations.

Suppose that $f_\alpha$ enters $b$ at the third place.  This
implies in particular that
$$0 < r_\alpha^3(b) - r_\alpha^1(b) = n_2 - m_3-m_1 = u - m_1,$$
and so $u> m_1 \geq 0.$

Moreover this insertion increases $m_3$ by 1.  The latter can be
achieved without changing the remaining entries in $b =({\mathbf
m}_j,\ldots,{\mathbf m}_1)$ by decreasing $u$ and $a$ by 1 and
increasing $t$ by 1. One easily checks that all inequalities of
Proposition \ref {5.4} are preserved.

As a second example suppose that $f_{g\beta}$ enters $b$ at the
fourth place.  This implies in particular that
$$0 < r_{g\beta}^4(b) - r_{g\beta}^2(b)  =  m_3 + n_3 - n_4- n_2 =
t.$$

Moreover this insertion increases $n_4$ by 1.  The latter can be
achieved without changing the remaining entries in $b =({\mathbf
m}_j,\ldots,{\mathbf m}_1)$ by decreasing $t$ by 1 and increasing
$a$ and $a'$ by 1. Since already $u,v \geq 0$ and $t > 0$, it
easily follows that all inequalities are preserved.

\subsection{}\label{5.8}

Retain the conventions of \ref {5.6}.  To complete the proof of
Proposition \ref {5.4} it remains to show that
$B_J(\infty)^{\mathscr E} = \{b_\infty \}$. This obtains from the
following lemma which also implies that $B_J(\infty)$ is upper
normal.
  \begin{lemma} For all $b \in B_J(\infty)$, one has

(i)   $e_\alpha b = 0$, if and only if $a \geq 0, a + u - m_1 \geq
 0, m_5 = 0$,

(ii)  $e_{g\alpha}b = 0$, if and only if $a'\geq 0, a' + v - n_1
\geq 0, n_5 = 0$,

(iii) $e_\beta b = 0$, if and only if $s \geq 0, u + t + a'  = 0$,

(iv)  $e_{g\beta}b = 0$, if and only if $t \geq 0, a + a' + s + t
+ v = 0$.

\end{lemma}

 \begin{proof}  Suppose $e_\alpha b = 0$.  This means that there exists $j \in \mathbb N$ such that
 $e_\alpha$ enters $b$ at the $(2j + 1)^{th}$ place and that $m_{2j+1}=0$.  One has

 $$(r_\alpha^5,r_\alpha^3,r_\alpha^1)= (m_5, m_5 - a, m_5 - a + m_1 -
 u).$$

Now $u \geq 0$, so then $e_\alpha$ cannot enter at the first place
since this would mean that $m_1$ = 0.   Suppose $a < 0$.  Then
$e_\alpha$ enters $b$ at the third place implying that $0 = m_3 =
n_2- u = 0$. Since $n_5 \geq 0$, we obtain $v + s + t + a' \leq
0$, and since $a < 0$, this contradicts $(iv)$ of \ref {5.4}.
Thus $a \geq 0$ and so $e_\alpha$ enters at the fifth place.  All
this gives $(i)$. The proof of $(ii)$ is practically the same.

Suppose $e_\beta b = 0$.  This means that there exists $j \in
\mathbb N^+$ such that $e_\beta$ enters $b$ at the $2j^{th}$ place
and that $m_{2j} = 0$. One has
$$(r_\beta^6, r_\beta^4,r_\beta^2)  = (0, u + t +a', u + t + a' -
s).$$

Recall that $u + t + a' \geq 0$.  Suppose $s < 0$. Then $e_\beta$
enters $b$ at the second place implying $m_2 = 0$.    Since $m_5
\geq 0$, we obtain $v + t + a = 0$, and since $s < 0$, this
contradicts $(iv)$ of \ref {5.4}. Thus $s> 0$. If $u + t + a' >
0$, then $0 = m_4 = n_2 - v - s$.  Yet $n_5 \geq 0$, and this
forces a contradiction. All this gives $(iii)$.  The proof of
$(iv)$ is practically the same.
\end{proof}

\subsection{}\label{5.9}

Let $J'$ be the second non-redundant sequence described in \ref
{4.2}. We sketch briefly how to obtain a crystal isomorphism
$\varphi:B_J(\infty)\iso B_{J'}(\infty)$.

Let us use $F_\alpha^{\mathbf m}$ to denote
$f_\alpha^mf_{g\alpha}^n$, when $\mathbf m=(m,n)$ and $E_\alpha
b=0$ to mean $e_\alpha b = e_{g\alpha}b   = 0$.  Similar meanings
are given to these expressions when $\beta$ replaces $\alpha$.

Then  $f:=F_\alpha^{\mathbf m_1}F_\beta^{\mathbf
m_2}F_\alpha^{\mathbf m_3}F_\beta^{\mathbf m_4}F_\alpha^{\mathbf
m_5}$ is said to be in normal form if

$$E_\alpha F_\beta^{\mathbf m_2}F_\alpha^{\mathbf m_3}F_\beta^{\mathbf m_4}F_\alpha^{\mathbf
m_5}b_{\infty}=0,E_\beta F_\alpha^{\mathbf m_3}F_\beta^{\mathbf
m_4}F_\alpha^{\mathbf m_5}b_{\infty}=0,\ldots, E_\alpha
b_{\infty}=0.$$ \

Let $\mathscr F_0 \subset\mathscr F$ denote the subset of elements
having normal form. Obviously for each element of $B_J(\infty)$
can be written as $fb_\infty$ , for some unique $f \in \mathscr
F_0$. (Nevertheless it should be stressed that the definition of
normal form depends on a choice of reduced decomposition.)

Take $f$ as above with $\mathbf m_j = (m_j,n_j):j \in \mathbb
N^+$, given by the expressions in $(i)-(iv)$ of \ref {5.4}.  One
checks that the condition for $f$ to have normal form exactly
reproduces the inequalities of Proposition \ref {5.4}. Moreover
$fb_\infty$ takes exactly the form of the element $b$ occurring in
the proposition, namely $(\mathbf m_5,\ldots,\mathbf m_1)$. In the
classical Cartan case this is essentially Kashiwara's algorithm
(cf \cite[see after Cor. 2.2.2]{Ka2} for computing $B_J(\infty)$;
but it is not too effective as each step becomes increasingly
arduous. Again in our more general set-up it was not a priori
obvious that such a procedure would work. The explanation of its
success obtains from Section 7.

\subsection{}\label{5.10}

Now let $b_\infty'$ denote the canonical generator of
$B_{J'}(\infty)$ of weight zero.   We define the required crystal
map $\varphi$ of \ref {5.9} by setting $\varphi(fb_\infty) =
fb_\infty'$, for all $f \in \mathscr F_0$.  One checks that $f$
still has normal form with respect to $b_\infty'$. (This is a
slightly different calculation to that given in \ref {5.9}; but
still uses exactly the same inequalities of Proposition \ref
{5.4}.) It follows that $\varphi$ is injective. Repeating this
procedure with $J$ and $J'$ interchanged we obtain crystal
embeddings $B_J(\infty)\hookrightarrow
B_{J'}(\infty)\hookrightarrow B_J(\infty)$. Since these preserve
weight and weight subsets have finite cardinality, they must both
be isomorphisms, as required.

\subsection{}\label{5.11}

Recall \ref {3.5}.   Then just as in \cite {J2}, 5.3.13, we may
define a singleton highest weight crystal $S_\lambda = \{
s_\lambda \} $ of weight $\lambda$ with
$\varepsilon_\alpha(s_\lambda) = -\alpha^\vee(\lambda )$, for all
$\lambda \in P^+$ and one checks that $B(\lambda) = \mathscr
F(b\otimes s_\lambda)$ is a strict subcrystal of $B(\infty)\otimes
S_\lambda$. The upper normality of $B(\infty$) implies that
$B(\lambda)$ is a normal crystal (see \cite {J1}, 5.2.1, for
example) and moreover its character can be calculated in a manner
analogous to the case of a classical Cartan matrix (see \cite
{J1},6.3.5, for example) using appropriate Demazure operators.
This leads to the character formula for $B(\infty)$ alluded to in
\ref {5.2}. However one may now begin to suspect that $B(\infty)$
is itself just a Kashiwara crystal for a classical Cartan matrix
of larger rank. This is what we shall show in Sections 6 and 7.

\section{The extended Weyl group for the Pentagonal Crystal}
\subsection{}\label{6.1}

Take $\pi$ and $C$ as in \ref {5.1}.  One checks that {W} applied
to $\pi$ generates as a single orbit the set $\Delta_s$, called
the set of short roots, given by the formulae below
$$\Delta_s=\Delta_s^+\sqcup\Delta_s^-, \Delta_s^-=-\Delta_s^+,\Delta_s^+=
\{\alpha,\alpha+g\beta,g\alpha+g\beta,g\alpha+\beta,\beta\}.$$

Set $\Delta_\ell=g\Delta_s$, called the set of long roots. They
are collinear to the short roots. Finally set
$\Delta=\Delta_s=\Delta_s\sqcup\Delta_\ell$.

\subsection{}\label{6.2}

At first sight it may seem that the introduction of $\Delta_\ell$
is quite superfluous, yet it is rather natural from the point of
view of the theory of crystals.   Indeed we have already noted
that we need both short and long simple roots to interpret the
Kashiwara tensor product and then to obtain $B(\infty)$.   Again
the analysis of \ref {5.11} implies that
$$ch B(\infty) = \prod_{\gamma\in \Delta^+}(1-e^{-\gamma})^{-1}$$

As promised we shall eventually prove the above formula by
slightly different means.    However for the moment we note that
it may in principle be obtained from the inequalities in \ref
{5.4}. Here we have ten parameters whose values define a subset of
$\mathbb N^{10}$. It is possible to break this subset into smaller
"sectors" in which each of the corresponding terms become a
geometric progression. In general this is a rather impractical
procedure and indeed even in the present case the number of
sectors runs into several thousand. However one can at least
deduce that $ch B(\infty)$ must be a rational function of the
variables $e^\alpha,e^{g\alpha},e^\beta,e^{g\beta}$. Moreover
using \ref {5.8}, one may deduce that the subset $B^\alpha$ of
$B_J(\infty)$ defined by the condition $\mathbf m_1 = 0$ has an
$s_\alpha$ invariant character. Yet $B_J(\infty)\cong
B^\alpha\otimes B_\alpha$, and so $ch B_J(\infty)$ is invariant
with respect to an obvious translated action of $s_\alpha$.
Through the isomorphism $B_J(\infty)\iso B_{J'}(\infty)$,
described in \ref {5.9}, \ref {5.10}, it follows that the common
crystal admits a formal character which is invariant under the
translated action of $W$. A slightly more refined analysis shows
it to be invariant under a translated action of the augmented Weyl
group $W^a$.

In the present "finite" situation it is more efficient to use the
Demazure operators as alluded to \ref {5.11}.  However in general
there is no known combinatorial recipe to obtain an expression for
$ch B(\infty)$ as an (infinite) product.  This is mainly because
there is no known meaning to imaginary roots, outside their Lie
algebraic definition. In the non-symmetrizable Borcherds case such
a product formula is not even known, even though $B(\infty)$ can
be defined and its formal character calculated \cite[Thm. 9.1.3]
{JL}.

\subsection{}\label{6.3}

Recall \ref {3.9}.   It is rather easy to compute the image of a
generator of $W^a$ on a given root, always remembering however
that its elements are only  $\mathbb Z$-linear maps.   We need
only do this for the $s_{\alpha,i} : i = 1,2$, since the action of
the remaining elements can be obtained by $\alpha,\beta$
interchange. The result is given by the following

\begin{lemma} \

(i)   $s_{\alpha,1}$ stabilizes $g\alpha,\beta, g\alpha + \beta$,
and interchanges elements in each of the pairs $(\alpha,-\alpha);
(\alpha+ g\beta,g\beta);(g\alpha + g\beta, g^2\alpha
+g\beta);(g\alpha+ g^2\beta, g^2(\alpha + \beta))$.

(ii)  $s_{\alpha,2}$ stabilizes $\alpha,g^2(\alpha+\beta), g\alpha
+ g^2\beta$, and interchanges elements in each of the pairs
$(g\alpha,-g\alpha); (\alpha+ g\beta,g^2\alpha+g\beta);(g\alpha +
\beta, \beta);(g\beta, g(\alpha + \beta))$.
\end{lemma}

\subsection{}\label{6.4}

We see that $s_{\alpha,1}$ can interchange long and short roots.
Since is $\Delta$ just two $W$ orbits it follows that $\Delta$ is
a single $W^a$ orbit. The stability of $\Delta$ under $W^a$ came
to us as a surprise.  It truth depends on the particular basis we
have chosen for $M$.

\subsection{}\label{6.5}

A further surprise is the following.   Let $\pi=
\{\alpha_1,\alpha_2,\alpha_3,\alpha_4\}$ be the simple roots of a
system of type $A_4$ in the standard Bourbaki \cite {B} labelling.
Make the identifications $\alpha_1=\alpha,\alpha_2 =
g\beta,\alpha_3=g\alpha,\alpha_4=\beta$. Then the relations in
\ref {6.3} allow us to make the further identifications
$s_{\alpha_1} = s_{\alpha,1},s_{\alpha_2}= s_{\beta,2},
s_{\alpha_3} = s_{\alpha,2},s_{\alpha_4} = s_{\beta,1}$. We
conclude that $W^a$ is actually the Weyl group $W(A_4)$ of the
system of type $A_4$ with the augmented set of simple roots
$(\alpha,g\beta,g\alpha,\beta)$ being the set of simple roots of
this larger rank system. However notice that $s_\alpha$ (resp.
$s_\beta$) is not a reflection in this larger system; but rather a
product of commuting reflections, namely:
$s_{\alpha_1}s_{\alpha_3}$ (resp. $s_{\alpha_2}s_{\alpha_4})$.
Precisel;y what we obtain is the following

\begin{lemma}

\

(i) \   The map
$\{s_{\alpha,1},s_{\beta,2},s_{\alpha,2},s_{\beta,1}\}
\rightarrowtail \{
s_{\alpha_1},s_{\alpha_2},s_{\alpha_3},s_{\alpha_4} \}$ extends to
an isomorphism $\psi$ of $W^a$ onto $W(A_4)$ of Coxeter groups. In
particular
$$\psi(s_\alpha)=s_{\alpha_1}s_{\alpha_3},\quad \psi(s_\alpha)=s_{\alpha_2}s_{\alpha_4}.$$

(ii) \   Set $\widehat{W}= W^a \iso W(A_4)$.  The map
$\{\alpha,g\beta,g\alpha,\beta \} \rightarrowtail
\{\alpha_1,\alpha_2,\alpha_3,\alpha_4 \}$ extends to an
isomorphism of  $\mathbb Z \widehat{W}$ modules.

\end{lemma}

\textbf{Remark}. Thus the root diagram of $A_4$ can be drawn on
the plane.  This is the beautiful picture presented in Figure 1,
which illustrates a tiling using two triangles based on the Golden
Section.  It not only admits pentagonal symmetry coming from $W$;
but also a further "hidden" symmetry coming from $\widehat{W}$.
Similar considerations apply to weight diagrams for $A_4$.  In
particular the defining (five-dimensional) module for $A_4$ gives
rise to what we call a zig-zag triangularization of the regular
pentagon, see Figure 2. The fact that the longer to shorter length
ratio of the triangles obtained is the Golden Section now follows
from the above lemma! Moreover the angles in the triangles are
thereby easily computed. In Section 8 we describe the
generalization of the above lemma to type $A_{2n}$ and its
consequences for aperiodic tiling based on the $n$ triangles
obtained from appropriate triangularizations of the regular
$(2n+1)$-gon (see Figure 2).

\subsection{}\label{6.6}

There is a relation of the Golden Section to the dodecahedron
which goes back to ancient times.  As a consequence it is not too
surprising that our root system can be obtained from the vertices
of the dodecahedron by projection onto the plane of one of the
faces (see Figure 1). This is an orthogonal projection as the
remaining co-ordinate is normal to the face in question. Let us
describe the co-ordinates of the vertices of dodecahedron
$\{\alpha_1,\alpha_2,\alpha_3,\alpha_4\}$ which thus project onto
$\{\alpha,g\beta,g\alpha,\beta\}$.

The planar co-ordinates of $\alpha$ and $\beta$ which generate our
pentagonal system can be taken to be
$$\alpha =(1,0),\quad \beta = (-g/2,\sqrt{1-g^2/4}).$$ View these
as two vertices, called $\alpha_1$ and $\alpha_4$, of one face
$f_1$ of the dodecahedron. Now project the dodecahedron onto
$f_1$. Of course this takes all the vertices of the dodecahedron
onto the plane defined by $f_1$. Exactly one of these becomes
collinear to $\alpha$ (resp. $\beta$) and this with relative
length of exactly g. Call it $\alpha_3$ (resp.$\alpha_2$). Now the
perpendicular distance of $f_1$ to its opposite face $f_2$ is
exactly $(g+1)$. We take this normal direction to these two faces
as defining a third co-ordinate fixed to be zero at the centre of
the dodecahedron.  Thus the value of this co-ordinate on $f_1$
equals $(g+1)/2$, whilst its value on $\alpha_3$ and on $\alpha_2$
turns out to be $(g-1)/2$. We conclude that the co-ordinates of
these four vertices of the dodecahedron are given by
$$\alpha_1=(1,0,(g+1)/2),\quad \alpha_2=(-g^2/2,g\sqrt{1-g^2/4},(g-1)/2),$$
$$\alpha_3 =(g,0,(g-1)/2),\quad \alpha_4 = (-g/2,\sqrt{1-g^2/4},(g+1)/2).$$
These co-ordinates allow one to calculate the scalar products
between the above vertices.   One finds that
$(\alpha_i,\alpha_i)=3(g+2)/4$, for all $i=1,2,3,4$, whilst
$(\alpha_1,\alpha_2)=(\alpha_3,\alpha_4) = -1/2 -g/4
=-(\alpha_1,\alpha_4)$ and
$(\alpha_1,\alpha_3)=(\alpha_2,\alpha_4)=5g/4
=-(\alpha_2,\alpha_3)$. From this one obtains a further fact -
there is no orthogonal projection of the set $\pi$ of simple roots
of $A_4$ onto these vertices of the dodecahedron. Indeed the above
vectors have all the same lengths as do also their pre-images
$\hat{\alpha_i}:i=1,2,3,4$. Thus for some $a \in \mathbb R$ the
fourth co-ordinate must be $a$ or $-a$ in each element of $\pi$.
Since $(\hat{\alpha_1},\hat{\alpha_3})=0$, this already forces
$a^2=5g/4$ and then that $(\hat{\alpha_2},\hat{\alpha_3})=-5g/2$
and $(\hat{\alpha_1},\hat{\alpha_1}=2g+3/2$, giving the
contradiction $g=1/2$.

\section{Beyond the Pentagonal Crystal}
\subsection{}\label{7.1}

The aim of this section is to show that the pentagonal crystal
described above and coming from the Golden Section is in fact a
manifestation of the "ordinary" Kashiwara crystal in type $A_4$
with respect to special reduced decompositions of the longest
element of the Weyl group of the latter.  The construction we give
can be immediately generalized to the largest zeros of the
$(2n+1)^{th}$ Chebyshev polynomial, though as we shall see there
is a slight difference when $2n+1$ is not prime.  We start with
three easy (read, well-known) facts.

\subsection{}\label{7.2}

For our first easy fact, let
$\pi=\{\alpha_1,\alpha_2,\ldots,\alpha_{2n}\}$ be the set of
simple roots for a system of type $A_{2n}$ in the usual Bourbaki
labelling \cite[Appendix]{B}.  Set
$s_i=s_{\alpha_i}:i=1,2,\ldots,2n$. These are Coxeter generators
of the Weyl group for type $A_{2n}$, and which we view as
elementary permutations.  Set $s_\alpha = s_1s_3\ldots s_{2n-1},
s_\beta = s_{2n}s_{2n-2}\ldots s_2$, which are involutions. Take
$i \in \{1,2,\ldots,n\}$.  One checks that $s_\alpha
s_\beta(2i+1)=s_\alpha(2i)=2i-1$, whilst $s_\alpha
s_\beta(2i)=s_\alpha(2i+1)=2i+2$.  Thus $s_\alpha s_\beta$ is a
(2n+1)-cycle.  We conclude that $<s_\alpha,s_\beta>$ is the
Coxeter group isomorphic to $\mathbb Z_n \ltimes \mathbb Z_2$ as
an abstract group.  Its unique longest element $w_0$ has two
reduced decompositions namely $(s_\alpha s_\beta)^n s_\alpha$ and
$s_\beta(s_\alpha s_\beta)^n$. Observe further that $s_\beta
(s_\alpha s_\beta)^n(2i) =s_\beta (2(n+1-i)+1) = 2(n+1)-2i$,
whilst $s_\beta (s_\alpha s_\beta)^n(2i+1)=s_\beta
(2(n-i))=2n+2-(2i+1)$. Thus $w_0$ is also the unique longest
element of the Weyl group in type $A_{2n}$. Moreover substituting
the above expressions for $s_\alpha, s_\beta$, it follows that
$w_0$ has length $\leq (2n+1)n$ and so the resulting expressions
for $w_0$ as products of elementary permutations are again reduced
decompositions in the larger Coxeter group.

\subsection{}\label{7.3}

Our second easy fact concerns the factorization of the Chebyshev
polynomials $P_{2n}(x):n \in \mathbb N^+$.  Define the polynomials
$$Q_n(x):=P_n(x)-P_{n-1}(x):n \in \mathbb N.\eqno (*) $$  These satisfy the
same recurrence relations as the $P_n(x)$; but with different
initial conditions which are now $Q_{-1}=1,Q_0=1$.  A few examples
are
$$Q_1=x-1,\ Q_2=x^2-x-1,\ Q_3=x^3-x^2-2x+1,\
Q_4=x^4-x^3-3x^2+2x+1.$$  Of these just $Q_4$ factorizes over
$\mathbb Q$ giving $Q_4=(x-1)(x^3-3x-1)$. It is the first case
when $2n+1$ is not prime.

\begin{lemma}   For all $n \in \mathbb N$ one has

$(i)_n\quad  Q_n(x)Q_n(-x) = (-1)^nP_{2n}(x),$

$(ii)_n\quad Q_n(x)Q_{n-1}(-x) = (-1)^{n-1}P_{2n-1}(x)+(-1)^n,$

$(ii)_n\quad Q_{n+1}(x)Q_{n-1}(-x) = (-1)^{n+1}P_{2n}(x)+(-1)^nx.$
\end{lemma}
\begin{proof}  The cases for which $n=0$ are easily checked.  Then one checks
using the recurrence relations for the $P_n$ and $Q_n$ that
$(i)_{n-1}$ and $(ii)_{n-1}$ imply $(ii)_n$, that $(i)_{n-2}$ and
$(ii)_{n-1}$ imply $(iii)_{n-1}$ and finally that $(ii)_n$ and
$(iii)_{n-1}$ imply $(i)_n$.
\end{proof}

\subsection{}\label{7.4}

Our third easy fact describes when $Q_n(x)$ is irreducible over
$\mathbb Q$.

\begin{lemma}   For all $n \in \mathbb N^+$ one has

(i) The roots of $Q_n(x)$ form the set $\{2\cos (2t-1)\pi/(2n+1):
t \in \{1,2,\ldots,n \}\}$,

(ii) If \ $2m+1$ divides $2n+1$, then $Q_m(x)$ divides $Q_n(x)$,

(iii)  If \ $2n+1$ is prime, then $Q_n(x)$ is irreducible over
$\mathbb Q$.

\end{lemma}

\begin{proof}

$(i)$. By $(*)$ of \ref {2.2} and $(*)$ of \ref {7.3}, one has
$$(\sin \theta) Q_n(2\cos \theta)= \sin (n+1)\theta - \sin
n\theta.$$  Yet the right hand side vanishes for $\theta =
(2t-1)\pi/(2n+1): t \in \{1,2,\ldots,n\}$, since $\sin
(2n+1)\theta$ = 0 and $\cos n\theta = -\cos (n+1)\theta \neq 0$ .
Hence $(i)$. Then $(ii)$ follows from $(i)$ by comparison of
roots.

       $(iii)$.  Set $z=e^{i\theta}$, with $\theta = \pi/(2n+1)$.
By (i) the roots of $Q_n(2x)$ are the real parts of $z^{2t-1} :
t\in \{1,2,\ldots,n\}$.  Since $2n+1$ is assumed prime, these are
all $2(2n+1)^{th}$ primitive roots of unity.  There are therefore
permuted by the Galois group of $\mathbb Q[z]$ over $\mathbb Q$
and so are their real parts.  Hence they cannot satisfy over
$\mathbb Q$ a polynomial equation of degree $<n$.  Thus $Q_n(2x)$
is irreducible over $\mathbb Q$ and for the same reason so is
$Q_n(x)$.
\end{proof}

\subsection{}\label{7.5}

Return to the notation and hypotheses of 7.2.   Let $J$ be the
sequence of simple roots defined by the reduced decomposition
$(s_\alpha s_\beta)^ns_\alpha$ of $w_0$, namely
$J=\{\alpha,\beta,\ldots,\alpha\}$.  Set $g = 2\cos( \pi/(2n+1))$
which is the Golden Section if $n=2$.  Let $B_j$ denote the
elementary crystal corresponding to $j^{th}$ entry in $J$ counting
from the right. Thus $B_j$ is of type $\alpha$ (resp. $\beta$) if
j is odd (resp. even) and let $\mathbf m_j$ denote its entry. Set
$\mathbf m = \{\mathbf m_{2n+1},\ldots,\mathbf m_1\}$, which we
view as an element of $B_J$. Use the convention that $\mathbf m_j
=0, \forall j \notin \{1,2,\ldots,2n+1\}$. Then as in \ref {4.2}
successive differences of the Kashiwara function take the form
$$r_\alpha^{2j+1}(\mathbf m)-r_\alpha^{2j-1}(\mathbf m) = -\mathbf m_{2j+1} - \mathbf
m_{2j-1}+g \mathbf m_{2j}.\eqno{(*)}$$

When $\alpha$ is replaced by $\beta$ then the same relation holds
except that $2j+1$ is replaced by $2j$.

\subsection{}\label{7.6}

Retain the notation of \ref {7.5} but now interpret $\alpha$
(resp. $\beta$) in $J$ as the sequence
$\{\alpha_1,\alpha_3,\ldots,\alpha_{2n-1}\}$ (resp.
$\{\alpha_{2n},\alpha_{2n-2},\ldots,\alpha_2\}$), so now $B_J$
denotes the crystal which is a $2n(2n+1)$-fold tensor product of
elementary crystals in type $A_{2n}$. Let $B_{i,j}: j \in
\{1,2,\ldots,2n+1\}$ denote the elementary crystal $B_{\alpha_i}:i
\in \{1,2,\ldots,2n\}$ occurring in the $j^{th}$ place of $J$
counting from the right, let $m_{i,j}$ denote its entry and
$\mathbf m$ the element of $B_J$ they define. Then the successive
differences of Kashiwara functions take the form
$$r_{2i+1}^{2j+1}(\mathbf m) - r_{2i+1}^{2j-1}(\mathbf m)=
m_{2i+1,2j+1}-m_{2i+1,2j-1} + m_{2i,2j} + m_{2i+2,2j},\eqno
{(*)}$$ with $2i$ replacing $2i+1$ when $2j$ replaces $2j+1$.

\subsection{}\label{7.7}

We attempt to interpret $(*)$ of \ref {7.5} so that it becomes
$(*)$ of \ref{7.6}.   Since $Q_n(g)=0$ by $(i)$ of \ref {7.4} and
$Q_n$ is a degree $n$ monic polynomial with integer coefficients,
it follows that the ring $\mathbb Z[g]$ is a free $\mathbb Z$
module of rank n.  Let $\{g_i\}_{i=0}^{n-1}$ be a free basis for
$\mathbb Z[g]$ with $g_0=1$ and using the convention that $g_{-1}
= g_n =0$. Notice that we have shifted the labelling by $1$
relative to our convention in \ref {3.8} and that we are using the
same basis for all $\alpha \in \pi$.  Following \ref {3.8} we
write
$$r_\alpha^{2j+1}(\mathbf m)=\sum_{i=0}^{n-1}g_ir_{2i+1}^{2j+1}(\mathbf m),
\quad \mathbf m_{2j+1}=\sum_{i=0}^{n-1}g_im_{2i+1,2j+1},\quad
\mathbf m_{2j}=\sum_{i=1}^ng_{n-i}m_{2i,2j+1}. $$

After these substitutions and equating coefficients of $g_i$ in
$(*)$ of \ref {7.5} we obtain $(*)$ of \ref {7.6} given that
$gg_{n-i}=g_i+g_{i-1}$, equivalently that
$gg_i=g_{n-i}+g_{n-i-1}$, for all $i=0,1,\ldots,n-1$. Now set
$p_{2i}=g_i,p_{2i+1}=g_{n-i-1}$, for all $i=0,1,\ldots,[n/2]$.
Note that this implies $p_{-1}=0, p_0=1$ and $p_n=p_{n-1}$.  (We
also obtain $p_{n+1}=g_0$ for $n$ even but this we ignore.)  These
identifications give $p_{i+1}=gp_i-p_{i-1}$ for all
$i=0,1,\ldots,n-1$.  It follows that $p_i=P_i(g)$. In other words
$p_i$ is the $i+1^{th}$ Chebyshev polynomial evaluated at g. Since
the latter are monic polynomials of degee $i$, the
$\{p_i:i=0,1,\ldots,n-1\}$ form a free basis of $\mathbb Z[g]$. In
addition the relation $p_n=p_{n-1}$ becomes exactly the relation
$Q_n(g)=0$, precisely as required. Thus our goal has been achieved
and we have shown the

\begin{thm}  Let $C$ be the $2\times2$ Cartan matrix with $-2\cos
(\pi/2n+1):n \in \mathbb N^+$ as off-diagonal elements. Then the
crystal $B(\infty)$ defined through \ref{3.8} using the free basis
$\{P_i(g)\}_{i=0}^{n-1}$ of $\mathbb Z[g]$ is isomorphic to the
crystal $B(\infty)$ in type $A_{2n}$.
\end{thm}

\subsection{}\label{7.8}

From \ref {7.7} we obtain all the good properties of the
pentagonal crystal noted in Section 5, namely that it is upper
normal, independent of the choice of $J$ and satisfies the
character formula given \ref {6.2}.  We do \textit{not} obtain the
explicit description of $B_J(\infty)$ described in \ref {5.4}; but
we do obtain the justification of this description given by the
Kashiwara algorithm noted in \ref {5.9}.

\section{Weight Diagrams}
\subsection{}\label{8.1}

Recall that Lemma \ref{6.5} allows one to draw the weight diagrams
of $A_4$ in the plane and that furthermore from the defining
representation one naturally recovers the two triangles used in
(aperiodic) Penrose tiling. This leads to the following question.
Suppose we are given a Penrose tiling of part $P$ of the plane.
When do the vertices of $P$ viewed as a graph form a weight
diagram for $A_4$ ?

We note below that one may similarly draw the weight diagrams of
$A_{2n}$ in the plane and then one can further ask if it is
possible recover a family $\mathscr T_{2n+1}$ of triangles (see
\ref {8.4}) which lead to higher aperiodic tiling. Then one can
similarly ask which such tilings are weight diagrams.

Conversely given a weight diagram viewed as a set of points on the
plane can one join vertices to obtain a tiling in (part of) the
plane using just the elements of $\mathscr T_{2n+1}$ or possibly a
slightly bigger set ?  In fact the image of the weight lattice of
$A_{2n}$ for $n \geq 2$ is dense in the plane (for the metric
topology) and so the limit of weight diagrams takes on a fractal
aspect. Consequently one is certainly forced to supplement
$\mathscr T_{2n+1}$ using similar triangles which become smaller
and smaller by factors of $g^{-1}$.

These questions lead us to the following.  Observe that the
essence of Penrose tiling is that the two triangles involved are
self-reproducing up to similarity by factors of the Golden
Section. Here we show that this property naturally extends to
$\mathscr T_m$ for all $m \geq 3$.

\subsection{}\label{8.2}

We start with a generalization of Lemma \ref{6.5}. Fix a positive
integer $n$. Recall the notation of \ref {3.9} and \ref {7.7}. Let
$\pi := \{\alpha_1,\ldots,\alpha_{2n}\}$ be the set of simple
roots in type $A_{2n}$.  Notice that since now $M=\mathbb
Zg_0+\ldots+\mathbb Zg_{n-1}$, we should write
$s_{\alpha,i}\lambda=\lambda-\alpha^\vee(\lambda)_{i-1}g_{i-1}\alpha$
and
$s_{\beta,i}\lambda=\lambda-\beta^\vee(\lambda)_{i-1}g_{i-1}\beta$.

\begin{lemma} Set
$$\psi(s_{\alpha,i+1})=s_{2i+1},\psi(s_{\beta,i+1})=s_{2n-2i},$$
$$\psi(g_i\alpha)=\alpha_{2i+1},\psi(g_i\beta)=\alpha_{2n-2i},
\forall i=0,1,\ldots,n-1.$$

\

(i) \ $\psi$ extends to an isomorphism of $W^a$ onto $W(A_{2n})$
of Coxeter groups. In particular
$$\psi(s_\alpha)=\prod_{i=0}^{n-1}s_{2i+1},\quad \psi(\beta)=\prod_{i=0}^{n-1}s_{2n-2i}.$$

 Set $\widehat{W}= W^a \iso W(A_4)$.

 \

(ii) \   $\psi$ extends to a $\mathbb Z \widehat{W}$ module
isomorphism of $\mathbb Z[g]\alpha + \mathbb Z[g]\beta$ onto
$\mathbb Z\pi$.

\end{lemma}

\begin{proof} It suffices to verify that the $s_i\alpha_j: i,j
=1,2,\ldots,2n$, satisfy the correct identities.  Here we just
verify one example of a non-trivial case.  One has
$s_{2i+1}\alpha_{2i}=\psi(s_{\alpha,i+1}(g_{n-i}\beta))$, whilst
$$s_{\alpha,i+1}(g_{n-i}\beta)=g_{n-i}\beta-(g_{n-i}\alpha^\vee(\beta))_i\alpha=g_{n-i}\beta+(gg_{n-i})_i\alpha.$$
Yet by \ref {7.7} one has $gg_{n-i}=g_i +g_{i-1}$ and so the right
hand side above equals $g_{n-i}\beta+g_i\alpha$, whose image under
$\psi$ is $\alpha_{2i}+\alpha_{2i+1}$, as required.

\end{proof}

\subsection{}\label{8.3}

Retain the above notation.  We can only view $\mathbb Z[g]$ as a
lattice in the complex plane when $2n+1$ is prime, because by
Lemma 7.4 the monic polynomial satisfied by $g$, namely $Q_n$ is
irreducible over $\mathbb Q$ just when $2n+1$ is prime. We shall
avoid this difficulty as follows.

Recall that by \ref {7.4} the largest solution of the equation
$Q_n(g)=0$ is $g=2cos\pi/2n+1$. Let $\psi'$ be the composition of
$\psi^{-1}$ with evaluation of $g$ at the above value, which is of
course well-defined map of $\mathbb Z\pi$ into $\mathbb R$ which
is injective just when $2n+1$ is prime.

It is clear that the weight diagram of the defining
$2n+1$-dimensional representation of $\mathfrak {sl}(2n+1)$
becomes under $\psi'$ is exactly the regular $(2n+1)$-gon. Choose
its highest weight to be the fundamental weight $\varpi_1$
corresponding to $\alpha_1$. Let $v_0$ designate the corresponding
point $\psi^{-1}(\varpi_1)$ in the plane.  For all $i =0,1,\ldots
n-1$ join the vertex
$v_i:=\psi^{-1}(\varpi_1-\alpha_1-\alpha_2-\ldots-\alpha_i)$ to
$v_{i+1}:=\psi^{-1}(\varpi_1-\alpha_1-\alpha_2-\ldots-\alpha_{i+1})$.
This gives what we call a zig-zag triangularization of the regular
$(2n+1)$-gon. Such a triangularization is rather natural from the
point of view of representation theory. Let $T_i:i=1,2\ldots,n-1$
denote the triangle with vertices $\{v_{i-1},v_i,v_{i+1}\}$.  From
the above lemma or directly one can easily compute their angles
and edge lengths. Set $p_i=P_i(g)$ and let $\mathbf P_{2n+1}$ be
the monoid they generate in $\mathbb C$.   Observe that $T_i$ is
related to $T_{2n-i}$ by the parity transformation
$T\rightarrowtail T'$, which reverses the cyclic order of the
vertices.

\begin{cor} Take $i \in \{1,2,\ldots,n-1\}$. The angles (up to a multiple of $\pi$)
(resp. edge lengths scaled to one for the sides of the
$(2n+1)-gon$) in $T_i$ given in cyclic order starting from $v_i$
(resp. $v_{i-1}-v_i$ are $\{1/2n+1,i/(2n+1),(2n-i-1)/(2n+1)\}$
(resp. $\{p_{i-1},p_i,p_0\}$.
\end{cor}

\textbf{Remark}.   Notice that we may now give the conclusion of
Lemma \ref {8.2} the following aesthetically pleasing
presentation. Consider the extended Dynkin diagram of $A_{2n}$. We
may regard it as a regular $(2n+1)$-gon with vertices labelled by
the roots $\alpha_i : i =0,1,2,\ldots,2n$.  Then the distance from
$\alpha_0$ to $\alpha_i:i=1,2,\ldots,2n$, is $p_{i-1}$. Comparison
with Lemma \ref {8.2} shows that this is exactly the factor which
multiplies $\alpha$ or $\beta$ in the image of $\alpha_i$.

\subsection{}\label{8.4}

Unless otherwise specified a weight diagram (in the plane) will
mean the image under $\psi'$ of a weight diagram of $\mathfrak
{sl}(2n+1)$.  A weight triangularization of a weight diagram is
then defined to be a triangularization in which the vertices are
exactly the images of weights of non-zero weight subspaces.  For
example the zig-zag triangularization of the regular $(2n+1)$-gon
defined above is a weight triangularization of the weight diagram
of the defining representation.  It is not the only weight
triangularization possible and unless $n\leq 2$ other weight
triangularizations can lead to a different set of triangles (see
Figure 2). Nevertheless since there can be no vertices in the
interior of the $(2n+1)$-gon every edge must join two vertices on
the boundary and hence must come from a root.  In particular every
edge length must be some $p_i$. The additional triangles obtained
in this fashion (which include all possible isosceles triangles
with angles which are a multiple of $\pi/(2n+1)$ can be needed for
a general weight triangularization (see Figure 3).  In general the
set $\mathscr T_{2n+1}$ (or simply, $\mathscr T$ of all triangles
obtained from a weight triangularization of the $(2n+1)$-gon is
described by the following easy

\begin{lemma}  Suppose $T \in \mathscr T_{2n+1}$.  Then, up to multiples
of $\pi /(2n+1)$, the angles in $T$ are given by an (unordered)
partition of $2n+1$ into three non-zero parts.   The length of the
side opposite to the angle of size $\pi k/(2n+1)$ equals $p_{k-1}$
if $k\leqslant n$ and $p_{2n-k-1}$ if $k\geqslant n$.
\end{lemma}

\subsection{}\label{8.5}

Fix $m$ an integer $\geq 3$.  Label the vertices of the regular
$m$-gon by $\{0,1,2,\ldots,m-1\}$ in a clockwise order. Take
$i_1,i_2,i_3$ with $0\leq i_1<i_2<i_3\leq m-1$ and let
$T_{i_1,i_2,i_3}$ be the triangle whose vertices form the set
$\{i_1,i_2,i_3\}$.  We may also write $T$ as
$T\{i_2-i_1,i_3-i_2,i_1-i_3\}$, that is through its angle set
(omitting the multiple of $\pi/m$).  Of course we should not want
to distinguish such triangles which can be transformed into one
another by rotation; but it is not so obvious whether we should
equate triangles interrelated by parity. Indeed if the triangle
were a tile with all angles distinct, then its parity translate
could only be obtained by flipping it onto its "undecorated" side
! Thus we shall regard $T\{i,j,k\}, T\{j.k.i\},T\{k.i.j\}$ as the
same triangle T; but write $T'=T\{k,j,i\}$ for the triangle
obtained from $T$ through parity. In this convention $T=T'$, if
$T$ is an isosceles triangle.  Let $\mathscr T_m$ denote the set
of all such triangles.  This extends our previous definition for
$m$ odd.

\subsection{}\label{8.6}

Let $P$ be a polygon (in the Euclidean plane) and $p$ a
non-negative integer. Let $pP$ denote the polygon scaled by a
factor of $p$, being the empty set when $p=0$.  Suppose $P,P'$ are
polygons with share a side of the same length and which do not
overlap when fitted together along this side of common length.
Then we denote by $P*P'$ the resulting polygon.  One should of
course appreciate that this notation does not take into account
all possible fittings; but for us additional formalism will not
serve any purpose.  Again it may often be the case that a tiling
will violate this condition (see Figure 10).  This is in
particular true of the tilings obtained through the construction
of \ref {8.11}.

\subsection{}\label{8.7}

Retain the notation and conventions of \ref {8.5}. Extending \ref
{8.3} we set $T_i=T\{1,i,m-i-1\} \in \mathscr F_m$. Let
$p_{j-1}:j=1,2,\ldots,m-1$ denote the distance from the vertex $0$
to the vertex $j$. It is the length of a side opposite an angle of
size $j\pi/m$ in any element of $\mathscr T_m$.  Scale the
elements of $\mathscr T_m$, so that $p_0=1$ and set $p_m =0$.
Observe that $p_1=2\cos \pi/m$ and that $p_{j-1}=p_{m-1-j}$. Let
$\mathbf P_m$ denote the monoid generated by the
$p_i:i=1,2,\ldots,m-2$.

\begin{lemma}  For all $i \in \{1,2,\ldots,m-3\}$, one has
$T_i*T_{i+1}=p_iT_1$.  Moreover $p_1p_i=p_{i+1}+p_{i-1}$.
\end{lemma}

\begin {proof} The first part follows by joining the triangles
along their common side of length $p_0$.  Through similarity of
$p_iT_1$ with $T_1$, it implies the second part.

\end {proof}

\textbf{Remark}. Thus $p_i$ is the value of the $(i+1)^{th}$
Chebyshev polynomial $P_i$ at $x=2\cos\pi/m$.  Via $(*)$ of \ref
{2.2}, we further deduce that $x$ is the largest (real) root of
the equation $P_{[m/2]}=P_{[(m-3)/2]}$.

\

\textbf{ Example 1.} Let $T$ be the equilateral triangle of side
1. It is exactly the weight diagram of the defining representation
of $\mathfrak {sl}(3)$.   Then $T*T*T*T = 2T$ which is a weight
triangularization of the six-dimensional representation of
$\mathfrak {sl}(3)$.   Moreover as is well-known) a weight
triangularization of a weight diagram for $\mathfrak {sl}(3)$ can
be given in the form $T^{*n}$ for some positive integer $n$.  The
resulting tiling of the plane goes back to ancient times.

We remark that the relation $T*T*T*T = 2T$ holds for any triangle
$T$.  In this it suffices match up edges of the same length for
then the angles take care of themselves.  It results in a tiling
of the plane cut out by three infinite sets of parallel lines.
Surprisingly it is hardly ever seen or used - perhaps for only
technical reasons !

\

\textbf{ Example 2.} Consider the pair $T_1,T_2$ of triangles
given by \ref {8.3}.   Their sides have lengths $\{1,g,1\}$ and
$\{g,g,1\}$ respectively, where g is the Golden Section.  Their
areas $a_1,a_2$ satisfy $a_2/a_1=g$. We call it the Golden Pair.
Notice that $T_1*T_2=gT_1$ by the above and in addition that
$gT_1*T_2=gT_2$. Inductively we may generate the pair
$g^nT_1,g^nT_2, \forall n \in \mathbb N$ and moreover in $2^m$
possible ways, where m is the number of products. This fact is the
basis of one way of presenting Penrose tiling. (To be precise
Penrose constructed a "kite" as $T_1*T_1$ by joining these
triangles along their shortest edge and a "dart" as $T_2*T_2$ by
joining these triangles along their longest edge and then
considered \textit{certain} tilings obtained from kites and darts.
A useful discussion of this may be found in \cite {W}. We have
adopted the more prosaic tilings by the Golden pair as a
description of Penrose tiling.)

Just the two triangles of the Golden Pair do not suffice to give
all weight triangularizations. Thus although all root lengths are
all either $1$ or $g$, roots becoming collinear give vertices
separated by a distance of $g-1=g^{-1}$. Indeed if the highest
weight equals $\varpi_2+\varpi_3$ then the all three triangles
$\{T_1,g^{-1}T_1,g^{-1}T_2\}$ are needed for a weight
triangularization. In general such differences for example forces
one to smaller and smaller triangles until a weight triangulation
begins to resemble a fractal.

\

\textbf{Example 3} Take $n=3$.  In the zig-zag triangularization
of the regular $7$-gon one may replace the lines joining $v_3,v_4$
and $v_4,v_5$, by lines joining $v_3,v_6$ and $v_2,v_6$ and still
obtain a weight triangularization (see Figure 2).  This results in
a new triangle which we shall denote by $T_0$.  It is an isosceles
triangle with side lengths $\{p_2,p_1,p_1\}$ and angles (up to a
multiple of $\pi/7$) which are $\{2,2,3\}$.   Moreover if we
further replace the line joining $v_2,v_6$ by that joining
$v_3,v_4$ we obtain the "relation" $T_0*T_1 = T_2*T_3$. One might
compare the resulting "algebra" of triangles to the commutative
ring $\mathbb Z[x_0,x_1,x_2,x_3]/<x_0x_1-x_2x_3>$.  As is
well-known the latter is not freely generated.

Return for the moment to the Golden Pair $T_1,T_2$.  We noted in
Example 2 that it was possible using $*$ to generate all the
triangles in the set $\{pT_1,pT_2, \forall p \in \mathbf P_5\}$.
Moreover we were able to obtain a weight triangularization of the
root diagram of $\mathfrak {sl}(5)$ by using just $\{T_1,T_2\}$
scaled by a factor of $g^{-1}$.  Here the situation is more
complex. Thus in Figure 3 we illustrate a weight triangularization
of the root diagram of $\mathfrak {sl}(7)$.  It uses
$\{T_0,T_2,T_3\}$ scaled by a factor of $p_2^{-1}$.  The first
surprise is that we cannot just use the set of triangles coming
from a zig-zag triangularization of the weight diagram of the
defining representation,namely the set $\{T_1,T_2,T_2',T_2\}$ with
whatever scaling. In addition by spotting unions of triangles in
Figure 3 we obtain the following relations
$$T_0*p_2T_1=p_1T_3,\quad T_2*p_2T_1=p_1T_2,\quad T_2*p_1T_3=p_2T_2,\quad p_2T_2*T_3=p_2T_3,$$
$$p_2T_2*p_1T_3=p_2T_0, \quad T_0*T_0*T_2*T_3=p_1T_0.$$
(This last relation is more difficult to spot but is illustrated
in Figure 4.)

We conclude that all the triangles in the set
$\{pT_0,pT_2,pT_3,\forall p \in \mathbf P_7\}$ can be generated in
complete analogy with the case of the Golden Pair.  However a
further surprise is that we obtain some extra triangles through
the relation $T_0*T_2=(p_1/p_2) T_0$.  A further fact (though
perhaps less of a surprise) is that we cannot obtain $T_1$.  This
is because if $a_i$ denotes the area of $T_i$ then
$$a_0=(p_1+p_2)a_1,\quad a_2=p_2a_1,\quad a_0=(p_1+p_2)a_1,$$
so the area of $T_1$ is too small.  On the other hand by Lemma
\ref{8.6} and the above  $\{T_i: i=0,1,2,3\}$ generates
$\{pT_0,pT_1,pT_2,pT_7,\forall p \in \mathbf P_7\}$. Notice that
we also obtain the above relations when a given triangle $T$ is
replaced by $T'$. (Here only $T_2$ is affected.)  We conclude that
$\mathscr T_7$ generates the set $\{pT:p \in \mathbf P_7, T \in
\mathscr T_7 \}$.

\subsection{}\label{8.8}

The result in the Lemma \ref {8.7} may be viewed in another way
which makes its generalization to other elements of $\mathscr T_m$
immediate. Thus let $T_1 \in \mathscr T_m$ be the triangle defined
in \ref {8.7}.  Now join the vertex $0$ to any vertex $i:i \in
{2,3,\ldots,m-1}$. The resulting line cuts $T_1$ into two
triangles which are easily seen to be similar to two of those in
$\mathscr T_m$.  Moreover there are $m-3$ such decompositions
which come in pairs related by parity. These decompositions are
exactly those given in \ref {8.7}. This construction immediately
gives the result below.  Recall the notation of \ref {8.5}.

\begin{lemma}  Take $0<i<j<m$. Then for all $t:0<t<i$ one has
$$p_{j-i+t-1}T_{0,i,j}=p_{j-1}T_{0,t,j-i+t}*p_{j-i-1}T_{0,i-t,j}.$$

\end{lemma}

\begin{proof} Cut the given triangle by the line joining the
vertex $t$ to the vertex $j$.  Then compute angles and side
lengths through \ref {8.5}, {8.7}.

\end{proof}

\subsection{}\label{8.9}

When $n=3$, the above lemma gives all the pair decompositions
given in Example 3.  However it does not give the more tricky
decomposition of $p_1T_0$ into four parts. For $n>2$, the above
lemma is insufficient for our purposes because there are missing
$p_s$ factors on the left hand side.  This arises whenever $j-i>1$
or taking into account equivalences (under $W$) all sides of the
$T_{0,i,j}$ have length $>1$. However we can then make what we
call an inscribed decomposition of $T_{0,i,j}$. Up to a rotation
we can assume $i\leq min\{j, m-j\}$.  Then assume that $i>1$,
which is the "bad" case.

\begin{lemma}  For all $t:0<t<i$ one has
$$p_{2t-1}T_{0,i,j}=p_{t-1}T_{0,i,j}*p_{t-1}T_{0,i+t,j+t}*p_{t-1}T_{0,j-i+t,j}*p_{t-1}T_{0,i,j-t}.$$
\end{lemma}

\begin{proof} Draw a second triangle with vertices
$\{t,i+t,j+t\}$.  Join the vertices $\{u_1,u_2,u_3\}$ where
respectively the line $(0,i)$ meets $(t,i+t)$, the line $(i,j)$
meets $(i+t,j+t)$ and the line $(j,0)$ meets $(j+t,t)$. The
decomposes our original triangle into four triangles with vertices
$\{0,u_1,u_3\}$, $\{u_1,i,u_2\}$, $\{u_2,j,u_3\}$ and the
inscribed triangle $\{u_1,u_2,u_3\}$.  Compute angles and lengths
of edges starting from the periphery.  This computation is
illustrated in Figure 5.

\textbf{N.B.}  Notice that $T_{0,i,j}$ intersects its rotated twin
at six points and the sides of the inscribed triangle are obtained
by joining second successive intersection points.  However there
are two ways to do this and unless $T_{0,i,j}$ is equilateral only
the one specified in the proof works !  Namely one must join the
points of intersection of a given line of $T_{0,i,j}$ with the
corresponding line of its rotated twin.  If $T_{0,i,j}$ is
equilateral, the second choice corresponds to making a different
rotation.

\end{proof}

\subsection{}\label{8.10}

Yet we are still missing some $p_s$ factors, when $s$ is even. For
$m$ odd the first bad case occurs when $m=9$, which is also the
first case when $m$ is odd and not prime.   Indeed $\mathscr T_9$
admits just one exceptional triangle whose angles are not coprime
(up to the factor $\pi/9$), namely $T\{3,3,3\}$.  Let $\mathscr
T'_9$ denote its complement in $\mathscr T_9$.  Using Lemmas 8.8
and 8.9 one may verify that $\mathscr T'_9$ generates the set
$pT:p \in \mathbf P_9, T \in \mathscr T'_9$. This again allows one
to obtain higher Penrose tiling using nine triangles with angles
which are multiples of $\pi/9$.

In order to fill the above lacuna, consider the missing triangle,
namely $p_2T\{3,3,3\}$ whose decomposition into elements of
$\mathscr T_9$ is required in order to show that the set generated
by $\mathscr T_4$ contains $pT:p \in \mathbf P_9, T \in \mathscr
T_9$.  Here we may first recall that to decompose $p_iT\{3,3,3\}$
for $i=1,3$, we cut this triangle with a second one rotated by an
angle of $2\pi/9$.  This cuts the original triangle into three
triangles and an "internal" hexagon. The hexagon was then cut into
four triangles by joining second successive edges in either of the
two ways possible. The internal triangle is exactly the
"inscribed" triangle. Furthermore each of the three triangles in
this second set shares a common edge with one in the first set and
may be joined to it. The resulting decomposition of
$p_iT\{3,3,3\}:i=1,2$ is exactly what is described in \ref {8.9}.

We shall modify this procedure in two ways.  First rotate the
second triangle by just $\pi/9$.  Secondly cut the internal
hexagon into six triangles by joining third successive edges (that
is opposite edges). This cuts the original triangle into nine
smaller triangles.  Verification of angles and using the identity
$p_2^2=p_4+p_2+p_0$ we obtain the decompositon
$$p_2T\{3,3,3\}=T\{1,3,5\}^{*3}*T\{2,3,4\}^{*6},$$
illustrated in Figure 6.

One remark that if we rotate the second triangle by $3\pi/9$ then
its intersection with the first gives a well-known symbol beloved
by some. Moreover decomposition the internal hexagon by joining
opposite edges gives the decomposition
$$3T\{3,3,3\}=T\{3,3,3\}^{*9}.$$  This generalizes to give (via
Remark 1 of \ref {8.7}) the relation
$$nT=T^{*n^2}, \forall T \in \mathscr T, n \in \mathbb N^+.\eqno (*)$$

\subsection{}\label{8.11}

It comes as somewhat of a surprise that the above construction
does not generalize in the obvious fashion for all $n$, though the
identity $p_2^2=p_4+p_2+p_0$ is still valid for all $n\geq3$.  In
other words in order for example to decompose $p_2T: T\in \mathscr
T_n$ into nine triangles it is not appropriate to cut it with its
twin rotated by $\pi/2n+1$.  Nevertheless there is a way to
similarly decompose $p_2T:T\in \mathscr T_n$ into nine elements of
$T\in \mathscr T_n$.  This is illustrated in Figure 7 for the case
$n=5$ and $T=T\{3,3,5\}$.  It can be viewed as being obtained by
cutting $T$ with a second copy and decomposing the internal
hexagon as before; but the latter does not have vertices on the
same circle.

Finally we have illustrated the general decomposition of $p_2T:
T=T\{i,j,k\}\in \mathscr T_n; i,j,k \geq 3$ into nine triangles
symbolically in Figure 8. It is based on the identity $p_2p_i =
p_{i+2}+ p_i + p_{i-2}$, which holds for all $i: 2 \leq i \leq
m-4$. It is verified using $p_2=p_1^2-1$ and
$p_1p_i=p_{i+1}+p_{i-1}$.

In Figure 8 all nine triangles are drawn for convenience though
incorrectly as equilateral triangles. The correct angles are given
in each corner (as multiples of $\pi/m$). The reader will easily
discern a pattern and verify all the needed relations (which are
not entirely trivial - for example the angles opposite the edge
shared shared by a pair of triangles must be the same. Moreover at
external lines must all be straight ones.)

\subsection{}\label{8.12}  Finally we describe the decomposition
of $p_{t}T:T=T\{i,j,k\}\in \mathscr T_m; i,j,k \geq t+1$, into
$(t+1)^2$ triangles in $\mathscr T_n$. (We remark that the
construction does not specifically require t to be even.)
 First we need the
following preliminary. Recall that $p_i= P_i(g):i=0,1,\ldots,n, g
= 2\cos \pi/m$ and that $p_i:=p_{m-2-i}:(m-2) \geq i \geq n-1$
with $g$ being the largest real solution to the identity
$p_{[m/2]}=p_{[(m-3)]}$, namely $2\cos \pi/m$. Recall further that
$gp_i=p_{i-1}+p_{i+1}: 0<i<2n-1$.

\begin {lemma}  For all $i,t \in \mathbb N^+:
t\leq i\leq 2n-t-1$, one has
$$p_tp_i = \sum_{j=0}^t p_{i+t-2j}.$$
\end {lemma}

\begin {proof}  One has $$\begin{array}{lcl}
p_tp_i&=&(p_1p_{t-1} - p_{t-2})p_i,\\&=&p_{t-1}(p_{i+1} + p_{i-1}) - p_{t-2}p_i,\\

\end{array}$$
from which the assertion follows by induction on $t$.
\end {proof}

\subsection{}\label{8.13}

To describe the decomposition of $p_tT\{i,j,k\}$ it is perhaps
best to start with the cases $t=2$ and $t=3$.   The first has been
already been described symbolically in Figure 8.  The second case
decomposing $p_3T\{i,j,k\}: i,j,k >3, i+j+k=2n+1$ into 16
triangles in $\mathscr T_m$ is similarly described symbolically in
Figure 9. From these two cases the reader can easily figure out
the general solution for himself.  The result is described as
follows where we use $\prod$ instead of $*$.
\begin {prop} For all $m,t \in \mathbb N^+$ and $i,j,k > t$ with
$i+j+k = m$ one has

$$\begin{array}{lcl}p_{t}T\{i,j,k\}&=&\prod_{c=0}^{t}\prod_{r=0}^cT\{i+c-2r,j+2c-r,k-t+c+r\}*\\\\
&&\prod_{c=1}^{t}\prod_{r=1}^cT\{j+t-2c+r,k-t+c+r-1,i+c-2r+1\}.\\
\end{array}$$

\end {prop}

\begin {proof} Let us first explain the notation.  To begin with
$c$ (resp. $r$) labels columns from the top (resp. rows from the
left).  In the first column there is just one triangle, namely
$T\{i,j+t,k-t\}$.  Here the angles are given in clockwise order
starting from $i$ at the top.   Notice that the (upper) edges of
this triangle have lengths $p_{j+t-1}$ and $p_{k-t-1}$ as required
by Lemma \ref {8.12}.

The triangles appearing in the first product above are exactly
those which one vertex (with angle $i+c-2r$) "above" with the
remaining two vertices forming a "horizontal" line.  Those to the
extreme left (corresponding to $r=0$) have edges of lengths
$p_{j+2c}$ which together form the side of $p_{t}T\{i,j,k\}$
opposite to the angle of size $j$.  The sum of their edges equals
$p_{t}p_{j-1}$ via \ref {8.12}.  Similarly the sum of the edges of
those triangles corresponding to $r=c$ is just $p_{t}p_{k-1}$,
whilst the sum of the edges of these triangles in the last row
(corresponding to $c=t$) is just $p_{t}p_{i-1}$.

The triangles is the second sum are inverted relative to the
first.  Each share a common edge with a triangle in the first set.
One checks that the opposite angle sizes coincide.  Finally one
checks that any vertex has two (resp. 3,6) edges to it and for
which the resulting angle sizes are $\pi i$, $\pi j$ or $\pi k$
(resp. any three in cyclic order sum to $\pi$).  This means that
the large triangle  $p_{t}T\{i,j,k\}$ is cut into $(t+1)^2$
triangles in $\mathscr T_n$ by $3t$ lines with end points on its
edges exactly cutting the latter into the sums described in \ref
{8.12}.
\end {proof}

\subsection{}\label{8.14}

Let $<\mathscr T_m>$ denote the set of triangles generated by
$\mathscr T_m$ through $*$.  Combining \ref {8.8}, \ref {8.9},
\ref {8.13} we obtain the following

\begin {thm}  For all $n >1$ the set $pT: p \in \mathbf P_m, T \in \mathscr T_m$
is contained in $<\mathscr T_m>$.
\end {thm}

\subsection{}\label{8.15}

As we already noted the above inclusion may be strict (Example 3
of \ref {8.7}).   It may also be possible to combine triangles in
a different manner than that described using \ref {8.8}, {8.9},
{8.13}.  This already occurs for $m=9$ as illustrated in Figure
10. Nevertheless we have managed to accomplish the program
outlined in the last part of \ref {8.1}.  Notice that in virtue of
$(*)$ of \ref {8.10} we may obtain the stronger conclusion defined
by replacing $\mathbf P_m$ with $\widehat{\mathbf P}_n:= \{s
\mathbf P_m: s \in \mathbb N^+\}$.

\section{Fundamental domains, alcoves and the affine Weyl group.}
\subsection{}\label{9.1}

In the notation of \ref {2.2} define $\varpi_\alpha,\varpi_\beta$
to be the fundamental weights in $\mathfrak h^*$ given by
$\gamma^{\vee}(\varpi_{\zeta}) =
\delta_{\gamma,\zeta}:\gamma,\zeta \in \{\alpha,\beta\}$.  It is
immediate that if $\pi - \theta$ is the angle between $\alpha$ and
$\beta$ then $\theta$ is the angle between $\varpi_{\alpha}$ and
$\varpi_{\beta}$.  In particular the area between the lines they
define is a fundamental domain for the action of $W$ in $\mathfrak
h^*$.  Put another way the lines bordering this domain define
reflection planes and the group they generate has this domain as a
fundamental domain.  Notice that in this some integer multiple of
$\theta$ must equal $\pi$.

\subsection{}\label{9.2}

Now recall the remark in Example 1 of \ref {8.7} where we
described a tiling of the plane by a given triangle $T$. Consider
the group $W^{aff}$ generated by the reflection planes defined by
the three sides of the triangle.  One can ask if the given
triangle is a fundamental domain for the action of $W^{aff}$
acting on the plane.  It is clear that a necessary condition is
that this be true with respect to the subgroup generated by just
two reflection planes and the wedge shaped domain they enclose.
Now in \ref {9.1} we saw that this means that the angle they
define must have some integer multiple equal to $\pi$.  Thus we
must be able to write the angle set of $T$ in the form
$\{i\pi/m,j\pi/m,k\pi/m\}$ where $i,j,k$ divide $m$ and of course
sum to $m$.  We can of course further assume that the greatest
common divisor of $i,j,k$ equals one.  One easily checks the
well-known fact that this condition has just three solutions
$i=j=k=1$, $i=j=1,k=2$, $1=1,j=2,k=3$ corresponding to the
fundamental domains (called alcoves) in types $A_2,B_2,G_2$ under
the action of the affine Weyl group $W^{aff}$. Alcoves are not
disjoint; yet they intersect only in lower dimension and so we may
refer to their providing a decomposition of $\mathfrak h^*$ into
(essentially) disjoint subsets.

We conclude that the tiling of the plane described in Example 1 of
\ref {8.7} has rather rarely the special property described above
of being the translate of a fundamental domain.

\subsection{}\label{9.3}

The above apparently negative result nevertheless leads to the
following question.  Recall that for any simple Lie algebra
$\mathfrak g$ of rank n, the Cartan subspace $\mathfrak h$ admits
a decomposition into alcoves any one of which is a fundamental
domain for the action of the affine Weyl group.  We may therefore
ask if the decomposition into alcoves in type $A_{2n}$ leads to an
aperiodic tiling of the plane through the map $\psi'$ defined in
\ref {8.3}.

First we briefly recall how alcoves are obtained. (For more
details we refer the reader to \cite [Chap VI, Section 2]{B}. Fix
a system $\pi$ of simple roots and let $\{\alpha^\vee: \alpha \in
\pi\}$ (resp. $\{\varpi^\vee_\alpha:\alpha \in \pi\}$ be the
corresponding system of coroots (resp. fundamental coweights). Set
$Q^\vee = \mathbb Z\pi^\vee$. Let $\alpha_0$ be the highest root.
In type $A_n$ this is just $\alpha_1+\alpha_2+\ldots+\alpha_n$. By
a slight abuse of notation we let $s_{\alpha_0}$ denote the linear
bijection on $\mathfrak h$ defined by
$$s_{\alpha_0}(h) = h - (\alpha_0(h)-1) \alpha_0^\vee, \forall h \in \mathfrak h.$$
It is the reflection in the hyperplane in $\mathfrak h$ defined by
$\alpha_0(h) = 1$. The affine Weyl group $W^{aff}$ is the group
generated by the Weyl group for $\mathfrak g$ and $s_{\alpha_0}$.
It may be viewed as the Coxeter group with generating set
$s_{\alpha_0},s_{\alpha_1},\ldots,s_{\alpha_n}$. It is the
semi-direct product $W^{aff}= Q^\vee \ltimes W$.  Here one uses
the fact that $\alpha_0^\vee$ is the highest short root for the
dual root system and one checks that $Q^\vee = \mathbb
ZW\alpha_0^\vee$.

Let $m_i$ be the coefficient $\alpha_i$ in $\alpha_0$.  Then the
convex set in $\mathfrak h$ with vertex set
$\{0,\varpi^\vee_1/m_1,\varpi^\vee_2/m_2,\ldots,\varpi^\vee_n/m_n\}$
is a fundamental domain (alcove) for the action of $W^{aff}$ on
$\mathfrak h$. Since the $\varpi^\vee_i/m_i:i \in
\{1,2,\ldots,n\}$ are fixed points under $s_{\alpha_0}$, it is
transformed under $s_{\alpha_0}$ into the alcove with vertex set
$\{\alpha_0^\vee,\varpi^\vee_1/m_1,\varpi^\vee_2/m_2,\ldots,\varpi^\vee_n/m_n\}$.
The transformation of a given alcove under the reflections defined
by their faces (defined by $n-1$ vertices are similarly described.
In type $A_2$ their translates give the tessilation of the plane
by equilateral triangles as described in Example 1 of \ref {8.7}.

We remark that if $\mathfrak g$ is simply-laced then the
identification of $\mathfrak h$ with $\mathfrak h^*$ through the
Killing form identifies roots with coroots and weights with
coweights.  In addition in type $A$ the $m_i$ above are all equal
to one.  We shall use this identification and simplification in
the sequel.

\subsection{}\label{9.4}

In order to study the possible consequences of decomposition into
alcoves described in \ref {9.3} above for planar tiling we first
describe the images of the fundamental weights for type $A_{2n}$
under the map $\psi'$ of \ref {8.3}.  Fix $n \in \mathbb N^+$ and
recall the $p_i:i=0,1,\ldots,2n-1$ as described in \ref {8.12}.
Let $P(\pi)$ denote the set of (integer) weights in $\mathfrak
h^*$ relative to the choice of the set $\pi$ of simple roots.  Its
image under $\psi^{-1}$ is just the $\mathbb Z[x]/<Q_n(x)>$ module
generated by the integer weights relative to $\{\alpha,\beta\}$
denoted in \ref {3.5} by $P$.  Its image under $\psi'$ is obtained
by further evaluation of $x$ as $g = 2 \cos \pi/(2n+1)$.  (This
makes a difference only if $2n+1$ fails to be prime.)

\begin {lemma}  For all $i \in \{0,1,\ldots,n-1\}$, one has
$$\psi'(\varpi_{2i+1})=g_i\varpi_\alpha,\quad
\psi'(\varpi_{(2n-2i)})=g_i\varpi_\beta.$$

\end {lemma}

\begin {proof} Recall by Lemma \ref {8.2} that
$s_{2i+1}\psi(\lambda)=\psi(s_{\alpha,i+1}\lambda):
i=0,1,\ldots,n-1, \forall \lambda \in P$. Equivalently
$g_i\alpha^\vee_{2i+1}(\lambda)=(\alpha^\vee(\psi^{-1}(\lambda)))_i$,
for all $\lambda \in P(\pi)$. Similarly
$g_i\alpha^\vee_{2(n-i)}(\lambda)=(\beta^\vee(\psi^{-1}(\lambda)))_i$.
Hence taking $\lambda=\varpi_{2j+1}$, for $j \in
\{0,1,\ldots,n-1\}$ we obtain
$$(\alpha^\vee(\psi^{-1}(\varpi_{2j+1})))_i=\delta_{i,j}g_i,\quad
(\beta^\vee(\psi^{-1}(\varpi_{2j+1})))_i =0,\forall i,j
=0,1,\ldots,n-1.$$ This gives the first assertion. Taking
 $\lambda=\varpi_{2(n-j)}$ gives the second assertion.
\end {proof}

\subsection{}\label{9.5}

The above result already has a pleasing geometric feature worth
noting. First we recall the relation between the $g_i$ and the
chord lengths $p_i$ in the $(2n+1)$-gon, noted in \ref {7.7}.
Denote $\psi'(\varpi_i)$ simply as $x_i$, for all $i \in
\{1,2,\ldots,2n\}$. The image $T_f$ of the fundamental alcove
under $\psi'$ is the convex set over $\mathbb Q$ of $\{\{0\}, x_i
:i=1,2,\ldots,2n\}$.   It lies in the wedge enclosed by two
semi-infinite lines starting at the origin $\{0\}$ and forming an
angle of $\pi/2n+1$.  The triangle $T_i$ with vertex set
$\{\{0\},x_i,x_{i+1}\}$ for $i \in \{1,2\ldots,n-1\}$ is one of
those obtained by a zig-zag triangularization of the weight
diagram of the defining representation. (Moreover taken with  its
parity translate $T'$ having vertex set
$\{\{0\},x_{2n+1-i},x_{2n-i}\}$ for $i \in \{1,2,\ldots,n-1\}$
gives all such triangles.)  Since the distance between $\{0\}$ and
$x_1$ is $p_0=1$ the distance between $\{x_i,x_{i+1}\}$, for $i
\in \{0,1,\ldots,2n-1\}$ is again $1$.  Joining the points
$x_i,x_{i+1}:i=1,2,\ldots,n$ gives a triangularization $T_f$
viewed as an isosceles triangle $T\{1,n,n\}$ in $\mathscr
T_{2n+1}$ with vertex set $\{\{0\}.x_n,x_{n+1}\}$, by exactly the
isosceles triangles $p_{i-1}^{-1}T\{2n-2i+1,i,i\}:i
=1,2,\ldots,n-1$ in $\mathscr T_n$. Joining the points
$\{\{0\},x_{2n+1-i},x_{2n-i}\}:i=1,2,\ldots,n$ gives the parity
translated triangularization.  This result is a direct
generalization of the two possible ways to join the triangles
$\{T_1,T_2\}$ in the Golden pair to give $gT_1$.  Of course it is
also natural (read, tempting) to join in addition the points
$x_1,x_{2n+1-i}\}: i=1,2,\ldots,n-1$. These give the
triangularizations similar to those which appear in Figures 1,3 of
the root diagrams in types $A_4$ and $A_6$ except that they
involve weights rather than roots.

Augment the above notation by setting $x_0=x_{2n+1}=\{0\}$.  The
above result may be summarized by the following "shoelace"

\begin {lemma}  The distance between $x_i,x_{i+1}$ is independent
of $i \in \{0,1,\ldots,2n\}$.  Conversely the vertex set of the
image $T_f$ of the fundamental alcove can be obtained from the
cone with vertex $\{0\}$ and angle $\pi/2n+1$ by marking on
alternate sides of the cone equidistant points (namely the $x_i: i
=0,1,2,\ldots,2n+1$) starting (and ending) at $\{0\}$.
\end {lemma}

\subsection{}\label{9.6}

Recall that $\mathfrak g = \mathfrak {sl}(2n+1)$ and that we are
identifying $\mathfrak g$ with $\mathfrak g^*$ through the Killing
form. The presentation $W^{aff}= Q \ltimes W$ implies that the map
$\psi$ defined in Lemma \ref {8.2} commutes with the action of
$W^{aff}$. Moreover we may replace $\mathbb Z$ in its conclusion
by the field $\mathbb Q$ of rational numbers. Now view $\mathfrak
h_{\mathbb Q}$ as the Cartan subalgebra $\mathfrak
{sl}(2n+1,\mathbb Q)$. Let $\mathbb Q'$ denote the number field
$\mathbb Q[g] $, where as before $g = 2\bar{\cos\pi/(2n+1)}$. Then
$\psi'$ is a $\mathbb QW^{aff}$ map of $\mathfrak h_{\mathbb Q}$
onto $V_{\mathbb Q'}:=\mathbb Q'\alpha+\mathbb Q'\beta$, which is
injective if $2n+1$ is prime. In the latter case the (essentially
disjoint) decomposition of $\mathfrak h_{\mathbb Q}$ into alcoves
gives a corresponding (essentially disjoint) decomposition of
$V_{\mathbb Q'}$. This is less easy to interpret geometrically as
the image of each alcove is the convex set of its extremal points
over $\mathbb Q$, rather than over $\mathbb Q'$. As a consequence
the images of the alcove interiors appear to intersect when they
are naively drawn on the plane.

\subsection{}\label{9.7}

We have already given a description of the image $T_f$ of the
fundamental alcove in \ref {9.5}. Towards describing the images of
the remaining alcoves in $V_{\mathbb Q'}$ we proceed as follows.
First retain the conventions of \ref {9.7} and revert to our
notation of $\widehat{W}$ for the Weyl group of $\mathfrak
{sl}(2n+1,\mathbb Q)$.   Since the corresponding affine Weyl group
is the semidirect product of $\widehat{W}$ with the lattice of
roots, it follows from \ref {8.2} that we need only describe the
image of the fundamental alcove and its $\widehat{W}$ translates.
Indeed the images of the remaining alcoves will simply be
translates of the former by $\mathbb Z[g]\alpha +Z[g]\beta$. This
means in particular that there will be only finitely many image
types.

The image of the weights of the defining representation of
$\mathfrak {sl}(2n+1)$ is just the $(2n+1)$-gon with vertex set
$\{W\varpi_1\}$.  As is well-known the Grassmannian of the
corresponding module generates the remaining fundamental modules
(together with two copies of the trivial module).  Consequently
the set $\textbf{W} :=\{\widehat{W}\varpi_i:i=1,2,\ldots,2n\}$ is
exactly the set of all sums (different to zero) of distinct
elements of the weights of the defining representation and has
cardinality $2^{2n+1}-2$ (with no multiplicities).  Its image
under $\psi'$ is a union of non-trivial $W$ orbits and hence has
cardinality divisible by $2n+1$.  Thus $2n+1$ must divide
$2^{2n+1}-2$, if $\psi'$ separates the elements of $\textbf{W}$.
This generally fails unless $2n+1$ is prime.

Recall that $W=<s_\alpha,s_\beta>$ and for each $w \in W$ set
$r_w=\mathbb Q'w\alpha$ (resp. $p_w=\mathbb Q'w\varpi_\alpha$)and
$\textbf{R}$ (resp. $\textbf{P}$) their union.   It follows from
\ref {8.2} that the image under $\psi'$ of the roots of $\mathfrak
{sl}(2n+1)$ lie in $\textbf{R}$.  Thus one would have expected the
image of $\textbf{W}$ to lie in $\textbf{P}$.   This is false !
Rather we have

\begin {lemma} The image under $\psi'$ of $\{\widehat{W}\varpi_j: j
=1,2,\ldots,2n\}$ lies in the union of the lines $p_w: w \in W$ if
and only if $n=1,2$.
\end {lemma}

\begin {proof} Of course the case $n=1$ is trivial.   The case
when $2n+1$ is prime obtains from a simple numerical criterion.
Indeed in this case we can easily calculate the cardinality of the
intersection $\psi'(\textbf{W})\cap p_e$.  We claim that it equals
$2^{n+1}-2$ and hence that the cardinality of
$\psi'(\textbf{W})\cap \textbf{P}$ equals $(2^{n+1}-2)(2n+1)$.
Indeed if we take adjacent sums in the weights of the defining
representations using the convention that
$\varpi_{-1}=\varpi_{2n+1}=0$, we obtain the set
$S:=\{\varpi_{2i+1}-\varpi_{2i-1}:i=0,1,\ldots,n\}$ of cardinality
$n+1$ in which the $\varpi_{2i}:i=1,2,\ldots,n$, have cancelled
out. Clearly all ways of cancelling out these elements in taking
sums of distinct weights in the defining representation exactly
gives all possible non-trivial sums of distinct elements in $S$,
the number of which is $2^{n+1}-2$.

Now we have seen that the cardinality of $\psi'(\textbf{W})$
equals $2^{2n+1}-2$. Thus $\psi'(\textbf{W})\subset \textbf{P}$ if
and only if $(2^{n+1}-2)(2n+1)=2^{2n+1}-2$, that is when
$2^n+1=2n+1$, or $n=1,2$.

In the case when $2n+1$ is not prime it suffices to exhibit an
element of $\textbf{W}$ whose image under $\psi'$ lies strictly
inside the dominant chamber with respect to $<\alpha,\beta>$.  If
$n\geq 3$ the element $\varpi_1-\varpi_2+\varpi_4$ serves the
purpose.
\end {proof}

\subsection{}\label{9.8}

Though our result in \ref {8.14} rather banalizes aperiodic
tiling, the above lemma shows that the pentagonal system based on
the Golden Section has a special (or perhaps just simplifying)
feature.

Let us describe the images of all alcoves in the pentagonal case.
As before it is enough to describe those obtain from the
fundamental alcove through the action of $\widehat{W}$, since the
remaining ones are then obtained through translation by $\mathbb
Z[g]\alpha+\mathbb Z[g]\beta$, where $g$ is now the Golden
Section. Let us use $\mathscr I$ to denote the image set, namely
$W^aT_f$.

Each $T \in \mathscr I$ is the $\mathbb Q$ convex hull of its
vertex set, so it suffices to determine the latter.  Every vertex
must lie in $\textbf{W}\cup \{0\}\subset\textbf{P}$.  Every such
vertex set must contain $\{0\}$ which is a $W^a$ fixed point.
Moreover $W$ acts on $\mathscr I$ with every element having a
trivial stabilizer because this is already true for $W^a$. Of
course the action of $W$ as opposed to the action of $W^a$ is an
isometry. Since card$(W^a/W)=12$, we have just 12 vertex sets to
describe. The first forms the small regular pentagon $P_s$ of
which its remaining four vertices, two lie at distance $g^{-1}$ to
the origin, two at distance $1$ and all in $\textbf{W}$.  One may
remark that ten such small pentagons can be formed as expected.

The second forms the large regular pentagon $P_l$ of which its
remaining four vertices, two lie at distance $1$ to the origin,
two at distance $g$ and all in $\textbf{W}$.

The third is $T_0:=T_f$ itself and in this we note that it is the
apex of this isosceles triangle which lies at $\{0\}$. There are
four further isosceles triangles $T_i:i=1,2,3,4$ in $\mathscr I$
depending on which of its four remaining vertices lies at $\{0\}$.

Finally there are five rhombi $R_i:i=0,1,2,3,4$.  Each of these
have three sides of length $1$ and one side of length $g$ parallel
to one of the sides of length $1$. Its fifth vertex lies at the
intersection of its two diagonals. The latter is at a distance $1$
from two of its adjacent vertices on the boundary and at a
distance $g^{-1}$ from the opposite two. Thus as we already know
from the injectivity of $\psi'$ this "internal" vertex is not in
the $\mathbb Q$ convex set of the "external" vertices. Any one of
the five elements of its vertex set can be chosen to be $\{0\}$
and we use $R_0$ to designate the rhombus with $\{0\}$ as its
internal vertex.

As noted in \ref {9.3} for any alcove $A$ there is a unique
element $s$ of the affine Weyl group which fixes four of its
vertices. Thus there is a unique alcove, namely $sA$ which shares
with $A$ the face defined by the four fixed vertices. (Moreover
$s$ is the reflection defined by this face.) It follows that the
vertex set of $\psi'(A)$ shares exactly four elements with
$\psi'(sA)$. If $\psi'(A)$ is a large (resp. small) pentagon, then
$\psi'(sA)$ must be a rhombus (resp. triangle).  If $\psi'(A)$ is
rhombus then $\psi'(sA)$ can be a triangle or a large pentagon. If
$\psi'(A)$ is a triangle then $\psi'(sA)$ can be a rhombus or a
small pentagon.  For any alcove $A$ there are just five
reflections with the above property (each fixing four of the five
vertices).  Call them the allowed reflections of $A$. If $s$ an
allowed reflection of $A$ and $s'$ an allowed reflection of $sA$
different from $s$, we call $\{s',s\}$ an $A$ compatible pair.

\begin {lemma}  Let $A$ be an alcove, and $\{s',s\}$ an $A$ compatible pair.
Then $A$ and $s'sA$ share exactly three vertices. Conversely any
three vertices of $A$ define an $A$ compatible pair $\{s',s\}$ so
that $A$ and $s'sA$ share exactly the given vertices.
\end {lemma}

\begin {proof} Obviously $A$ and $s'sA$ share at least three
vertices and cannot share all five.  If they share four vertices
then there is a reflection $s''$ such that $s''A=s'sA$ forcing
$s''=s's$, which is impossible since all three are reflections.
The converse is obvious.
\end {proof}

\subsection{}\label{9.9}

Retain the above notation and hypotheses.  From the above lemma we
deduce that $\psi'(A)$ and $\psi'(s'sA)$ share exactly three
vertices.  The possible triangles they define each have as vertex
set any three vertices of the above four shapes, namely
$T,R,P_s,P_l$. From these we obtain a straight line segment
defined by two vertices with an additional vertex dividing it in
the Golden ratio, a Golden pair coming from $P_s$ and a second
Golden pair from $P_l$ inflated by a factor of $g$. The remaining
triangles coming from $R$ and $T$ are equivalent to these.
Conversely a triangularization of one of these shapes obtained by
joining its vertices defines a set of compatible pairs of the
given alcove.

The Golden pair $\{T_1,T_2\}$ coming from the triangularization of
the small pentagon has side lengths $\{1,g^{-1},g^{-1}\}$ and
$\{1,1,g^{-1}\}$ respectively in the above normalization.   Of the
remaining shapes the triangle can be decomposed into $T_1$ and two
copies of $T_2$ by lines joining its vertex set.  Such a
decomposition of $R$ and $P_l$ into members of this Golden pair is
(only) possible if we add extra vertices.

Our general idea is that covering $\mathfrak h^*$ by alcoves
coming from a fixed alcove $A$ should translate under a
\textit{particular} sequence of elements of the affine Weyl group
to give a covering of the plane by the triangles obtained Lemma
\ref {9.8}. Then aperiodicity can be introduced since different
sequences can be chosen.  However the simplest interpretation of
this procedure will not give a tiling since the resulting
triangles will overlap.  Worse than this the image under $\psi'$
of $P(\pi)$ is dense in the metric topology giving any such tiling
a fractal aspect.   We describe a rather ad hoc remedy to this
situation in the next section.

\subsection{Aperiodic tiling from alcove packing}\label{9.10}

Assume $n=2$, that is consider the pentagonal case.  Recall that
we have described the image $T_f$ under $\psi'$ of the fundamental
alcove.  It is an isosceles triangle with apex at $\{0\}$ with its
equal sides of length the Golden Section $g$ the third side being
of length $1$.  Finally it has two extra vertices on its equal
sides at distance $1$ from its apex.   Joining these vertices and
further just one of them to an opposite vertex on the base gives
it two possible triangularizations satisfying $T = T_1*T_2*T_2$.

To proceed we need the following

\begin {lemma} There exists a subset $\mathscr A$ of alcoves such
that the set $\psi'(A); A \in \mathscr A$ tiles the plane.
\end {lemma}

\begin {proof} Consider the triangle $T$ with vertices
$\{\psi'(0),\psi'(5\varpi_2),\psi'(5\varpi_4)\}$ and its parity
translate $T'$ with vertices
$\{\psi'(5\varpi_2+5\varpi_3),\psi'(5\varpi_2),\psi'(5\varpi_4)\}$.
Since $5\varpi_2,5\varpi_3$ both lie in $\mathbb Z\pi$, it follows
that $\psi'(\mathbb Z\pi)$ translates of $T$ and $T'$ tile the
fundamental chamber with respect to $W$. Hence their $W$
translates tile the whole plane. Recalling that the affine Weyl
group contains $\mathbb Z\pi$ and (the image of) $W$ it remains to
show that we can choose a subset $\mathscr A_0$ of alcoves whose
images tile $T$. We choose $\mathscr A_0$ so that the images of
its elements are again triangles, twenty-five in all. We must show
that all the elements of $\mathscr A_0$ are translates of the
fundamendal alcove under the affine Weyl group. Here we recall
that each of the images has exactly five vertices (coming as
explained previously as images of the vertices of the
corresponding alcoves).  Now we noted in \ref {9.8} that the
images of the $W^a$ translates of $T_f$ which are triangles come
in five $W$ orbits each determined by which vertex of the triangle
lies at $\{0\}$. Thus it remains to show that each of the above
twenty-five triangles has at least one vertex (and as it turns out
only one) which is a $\mathbb Z\pi$ translate of $\{0\}$.  This is
shown in Figure 11, the vertices in question being labelled by the
corresponding element of $P(\pi)$ which in each case the reader
will check lies in $\mathbb Z\pi$.  (This is the only non-trivial
and slightly surprising part of the proof.)
\end {proof}

\subsection{Aperiodic tiling from alcove packing, continued}\label{9.11}

It remains to give an (aperiodic) triangularization of each
$\psi'(A): A \in \mathscr A$.  This we shall do using the two
previous lemmas.   First starting from the fundamental alcove
generate the remaining alcoves in $\mathscr A$ by taking a
sequence of reflections in the affine Weyl group.   Of course we
get plenty of other alcoves but these we eventually discard. We
can add a predecessor to the fundamental alcove different to its
successor. Then every $A \in \mathscr A$ admits a predecessor
$A_-$ and successor $A_+$ obtained by a single generating
reflection $s$ (resp. $s'$) with $s \neq s'$. (Of course $A_-$ and
$A_+$ need not and will not belong to $\mathscr A$.)  By Lemma
\ref {9.8}, the pair $(s,s')$ determines three vertices
$v_1,v_2,v_3$ of $T:=\psi'(A)$, which we recall is a triangle. Now
as explained above we always join the vertices of $\psi'(A)$ lying
on its sides and breaking it into $T_2$ and a small rhombus $R$.
In addition there will be just two ways of writing $R$ as
$T_1*T_2$. If $v_1,v_2,v_3$ do not belong to the apex of $T$, then
they are three vertices of $R$ which when joined gives the
required decomposition.  Otherwise we take the four vertices of
$\psi'(A)\cap\psi(A_+)$ which now has exactly three vertices of
$R$ which we then join.  This gives the required (aperiodic)
tiling of the plane by the Golden Pair $T_1,T_2$ where the
different tilings correspond to different sequences in the affine
Weyl group successively running through the elements of $\mathscr
A$ and of course some discarded alcoves not in $\mathscr A$.

Of course all this is a bit of a swizzle, since in particular
rather many alcoves are discarded. Again a main point in the
construction is to find a subset of alcoves whose images form a
tiling of the plane. We found one example but certainly not all.
Nor can do we find all tilings of the plane that can be obtained
using the Golden Pair. In particular our construction leads to a
$1:2$ ratio in the contribution of the Golden Pair $T_1,T_2$,
whereas one might prefer to have a $1:1$ ratio. The latter could
be recovered by a equal weighted tiling of the plane using the
images of alcoves which are the triangle and small pentagon. Again
we could desist in joining the vertices of each $\psi'(A)$ lying
on its sides, which was done in all cases and so with no
aperiodicity. Then we would obtain a tiling by $gT_1,T_2$ with a
$1:1$ ratio.

A challenging problem would be obtain a three dimensional analogue
of the above construction. Namely for some simple root system
$\pi$, find a $\mathbb QW^{aff}$ linear map from $P(\pi)$ to
$\mathbb Q^3$ and a subset $\mathscr A$ of alcoves such that their
images form a packing of $\mathbb Q^3$. Then use the possible
words in $W^{aff}$ to obtain a multitude of (that is to say
aperiodic) packings in $\mathbb Q^3$.

\section{The Even Case}
\subsection{}\label{10.1}

We shall describe the analogue of Lemma \ref {8.2} when
$W:=<s_\alpha,s_\beta>\cong\mathbb Z_n \ltimes \mathbb Z_2$ with
$n$ even.   First we need to extend the factorization described in
Lemma \ref {7.3}.   Recall the Chebyshev polynomials $P_n(x)$
defined in \ ref {2.2}.   Set $$S_n(x)=P_n(x)-P_{n-2}(x), \forall
n \geq 1. \eqno{(*)}$$  One finds that $$S_2(x)=P_2(x)-1, S_1(x) =
P_1(x) =x, \quad S_{n+1}(x)=xS_n(x)-S_{n-1}(x), \forall n \geq
2.$$ It is convenient to set $S_0(x)=1$.

\begin {lemma}  For all $n \geq 1$, one has

$(i)_n \quad S_nP_{n-1}=P_{2n-1}$,

$(ii)_n \quad S_nP_n=P_{2n} +1$,

$(iii)_n \quad S_nP_{n+1} = P_{2n+1} +x$,

$(iv)_n \quad S_nP_{n-2}=P_{2n-2} - 1$.
\end {lemma}

\begin {proof}  From the above formulae and those in  \ref {2.2}, {8.12},
one easily checks the assertions for $n=1,2$.  For $ n \geq 2$ one
checks using the above recurrence relations for $S_n, P_n$, that
$(ii)_{n-1}, (iii)_{n-2}$ imply $(i)_n$, that $(i)_n, (iv)_n$
imply $(ii)_n$, that $(i)_n, (ii)_n$ imply $(iii)_n$ and that
$(i)_{n-1}, (ii)_{n-2}$ imply $(iv)_n$.
\end {proof}

\subsection{}\label{10.2}

The following is the analogue of Lemma \ref {7.4}.

\begin {lemma}  For all $n \in \mathbb N^+$ one has

(i)    The roots of $S_n(x)$ form the set $\{2\cos(2t-1)\pi/2n:
t\in \{1,2,\ldots,n\}\}$,

(ii)   Suppose $m$ is odd and divides $n$. then $S_d(x)$ divides
$S_n(x)$ with $d=n/m$.

(iii)  $S_n(x)$ is irreducible over $\mathbb Q$ if and only if $n$
is a power of $2$.

(iv)  Take $n>1$ and odd.  Then $x$ divides $S_n(x)$ and the
      quotient is irreducible over $\mathbb Q$ if and only if $n$ is
      prime.
\end {lemma}

\begin {proof}  By $(*)$ of \ref {2.2} and $(*)$ of \ref {10.1}
one has $$\sin \theta S_n(2\cos\theta)=\sin(n+1)\theta -
\sin(n-1)\theta.$$ The right hand side vanishes for $\theta =
(2t-1)\pi/2n : t \in \{1,2,\ldots,n\}$, whereas $\sin \theta \neq
0$.  Hence (i).  Then (ii) follows from (i) by comparison of
roots.  Set $z=e^{i\theta}$ with $\theta = \pi /2n$.  By (i) the
roots of $S_n(2x)$ are the real parts of $x^{2t-1}:t \in
\{1,2,\ldots,n\}$.  Since no odd number can divide a power of $2$
these are all primitive $4n^{th}$ roots of unity.  These are then
permuted by the Galois group of $\mathbb Q[z]$ over $\mathbb Q$
and so are their real parts.  Therefore they cannot satisfy over
$\mathbb Q$ a polynomial equation of degree $<n$.  Hence (iii).
The proof of (iv) is similar.
\end {proof}

\subsection{}\label{10.3}

As in \ref {2.2} we consider a $2 \times 2$ Cartan matrix with
off-diagonal entries $\alpha^\vee(\beta) = -y, \beta^\vee(\alpha)
=-1$, regarding $\{\alpha,\beta\}$ as a simple root system with
Weyl group given by $W=<s_\alpha,s_\beta>$ with the generators
being defined as in \ref {2.1}. Set $y=x^2$ with $x=2\cos\pi/m; m
\geq 3$. Previously we had considered the case when $m$ is odd,
say $m=2n+1$ and shown (Lemma \ref {8.2}) that this root system
together with its augmented Weyl group $W^a$ could be obtained
from a root system of type $A_{2n}$. Here we establish the related
results when $m$ is even.  This divides into two cases depending
on whether $m$ is divisible by $4$ or not.  This is not surprising
since the Cartan matrix is a system of type $B_2$ (resp. $G_2$) if
$m=4$ (resp. $m=6$).

\subsection{}\label{10.4}

Observe that $S_k(x)$ is a polynomial in $y=x^2$ if $k$ is even
which we write as $T_k(y)$. Again if $k$ is odd then $S_k(x)$ is
divisible by $x$ and $\frac{1}{x}S_k(x)$ is a polynomial in $y$
which we write as $T_k(y)$.

Set $m=4n$ in \ref {10.3}.   Take $x = 2\cos \pi/4n$.  Then
$S_{2n}(x)=0$ by Lemma \ref {10.2} and indeed is its largest
(real) root.   Let $\pi:=\{\alpha_1,\ldots,\alpha_{2n}\}$ be the
set of simple roots in type $B_{2n}$. Set
$s_i=s_{\alpha_i}:i=1,2,\ldots,2n$ be the corresponding set of
simple reflections in the Weyl group $W(B_{2n})$, with
$W=<s_\alpha,s_\beta>$ defined as in \ref {10.3}.

\begin {lemma}  Set $$\psi'(\alpha_{2(n-k)})=T_{2k}(y)\alpha, \quad
\psi'(\alpha_{2(n-k)-1})=T_{2k+1}(y)\beta, \forall k
=0,1,\ldots,n-1,$$
$$\psi'(\prod_{i=1}^n s_{2i})= s_\alpha, \quad \psi'(\prod_{i=1}^n s_{2i-1})= s_\beta.$$
Then $\psi'$ extends to a $\mathbb ZW$ epimorphism of $\mathbb
Z\pi$ onto $\mathbb Z[y]\alpha +\mathbb Z[y]\beta$.
\end {lemma}
\begin {proof} This is straightforward computation using the
relations in \ref {2.2} to apply successive products of
$s_\alpha,s_\beta$ to $\mathbb Z[y]\alpha +\mathbb Z[y]\beta$.
Compared to the corresponding products in the left hand side of
the second equation above applied to $\mathbb Z\pi$ this gives two
different ways to compute the left hand side of the first equation
above and one checks that both give the right hand side.
\end {proof}

\subsection{}\label{10.5}

By \ref {10.2}, the map $\psi'$ of \ref {10.4} is injective if and
only if $n$ is a power of $2$.   However we can make it injective
for all $n$ by reinterpreting $\mathbb Z[y]$ as the ring generated
by $y$ and satisfying exactly the relation $T_{2n}(y)=0$.  Of
course this ring has zero divisors if $n$ is not a power of $2$
and so cannot be embedded in $\mathbb R$.

Let us adopt the above interpretation of $\mathbb Z[y]$ so that
$\psi'$ becomes an isomorphism.   Then we can recover the
generating reflections of $W(B_{2n})$ by the same means as used in
type $A_{2n}$.   Observe that $M:=\mathbb Z[y]$ is a free $\mathbb
Z$ module of rank $n$.   However it now has \textit{two} natural
bases, namely $\{T_{2(n-i)}\}$ which we shall use to define the
$s_{\alpha,i}$ and $\{T_{2(n-i)+1}\}$ which we shall use to define
the $s_{\beta,i}$, for $i=1,2,\ldots,n$. More precisely
$s_{\alpha,i}$ is determined by $(*)$ of \ref {3.8} with $m_i:
m\in M$ defined by extending $\mathbb Z$ linearly the rule
$(T_{2(n-i)})_j=\delta_{i,j}:i,j=1,2,\ldots,n$ and similarly
$s_{\beta,i}$ is determined by $(*)$ of \ref {3.8} with $m_i: m\in
M$  defined by extending $\mathbb Z$ linearly the rule
$(T_{2(n-i)+1})_j=\delta_{i,j}:i,j=1,2,\ldots,n$. (This dichotomy
is essentially a result of replacing $x$ by $y$.) Then the
augmented Weyl group $W^a$ is defined as before to be the group
generated by the $s_{\alpha,i},s_{\beta,i}:i=1,2.\ldots,n$. With
these conventions we obtain the following

\begin {lemma} Set $$\psi'(s_{2i})=s_{\alpha,i}, \quad
\psi'(s_{2i-1})=s_{\beta,i}, \forall i=1,2,\ldots,n.$$

Then

(i) $\psi'$ extends to an isomorphism of $W(B_{2n})$ onto $W^a$.

Denote this common group by $\widehat{W}$.

(ii) $\psi'$ extends to a $\mathbb Z\widehat{W}$ module
isomorphism of $\mathbb Z\pi$ onto $M\alpha +M\beta$.
\end {lemma}

\begin {proof} For example
$$\alpha^\vee(\alpha_{2i-1})_i = -(x^2T_{2(n-i)+1})_i =
-(T_{2(n-i)+2}+T_{2(n-i)})_i=-1,$$ which gives
$s_{\alpha,i}\psi'(\alpha_{2i-1})=
\psi'(\alpha_{2i-1}+\alpha_{2i})$, as required.  Similarly for
example

$$\beta^\vee(\alpha_{2i})_i=
-(x^{-1}xT_{2(n-i)})_i=-(T_{2(n-i)+1}+T_{2(n-i)-1})_i=-1,$$ which
gives $s_{\beta,i}\psi'(\alpha_{2i-1})=
\psi'(\alpha_{2i-1}+\alpha_{2i})$, as required.
\end {proof}

\subsection{}\label{10.6}

Set $m=4n+2$ in \ref {10.3}.   Take $x = 2\cos \pi/(4n+2)$.  Then
$S_{2n+1}(x)=0$ by Lemma \ref {10.2} and indeed is its largest
(real) root.   Let $\pi:=\{\alpha_1,\ldots,\alpha_{2n+1}\}$ be the
set of simple roots in type $B_{2n+1}$. Let
$s_i=s_{\alpha_i}:i=1,2,\ldots,2n+1$ be the corresponding set of
simple reflections in the Weyl group $W(B_{2n+1})$, with
$W=<s_\alpha,s_\beta>$ defined as in \ref {10.3}.

\begin {lemma}  Set $$\psi'(\alpha_{2(n-k)+1})=T_{2k}(y)\alpha,\forall k
=0,\ldots,n, \quad \psi'(\alpha_{2(n-k)})=T_{2k+1}(y)\beta,
\forall k =0,\ldots,n-1,$$
$$\psi'(\prod_{i=1}^{n+1}s_{2i-1})= s_\alpha, \quad \psi'(\prod_{i=1}^n s_{2i})= s_\beta.$$
Then $\psi'$ extends to a $\mathbb ZW$ epimorphism of $\mathbb
Z\pi$ onto $\mathbb Z[y]\alpha +\mathbb Z[y]\beta$.
\end {lemma}

\begin {proof} Similar to that of Lemma \ref {10.4}.
\end {proof}

\subsection{}\label{10.7}

Take $n=1$ in \ref {10.6}.  Then $T_3(y)=y-3=0$ and so
$\alpha_3=\alpha$ and $\alpha_1=T_2\alpha=(y-2)\alpha = \alpha$.
Consequently $\psi'$ is not injective in this case.  Indeed
$<\alpha,\beta>$ is a system of type $G_2$ whilst $\pi$ is of type
$B_3$.  More generally $\psi'$ is never injective and this remains
true even if we interpret $M:=\mathbb Z[y]$ as the ring generated
by $y$ satisfying the relation $T_{2n+1}(y)=0$. The trouble is
that $T_{2n+1}$ is a polynomial of degree $n$ in $y$, whilst there
are $n+1$ simple roots which $\psi'$ maps to $M\alpha$.  To
recover injectivity we define $M_\alpha$ as the ring generated by
$y$ satisfying the relation $yT_{2n+1}(y)=0$, whilst we define
$M_\beta$ as the ring generated by $y$ satisfying the relation
$T_{2n+1}(y)=0$.  One checks that the conclusion of Lemma \ref
{10.6} remains valid when the target space of $\psi'$ is replaced
by $M_\alpha\alpha+M_\beta\beta$. (This is false if we also try to
make $M_\beta$ to be the ring generated by $y$ satisfying the
relation $yT_{2n+1}(y)=0$.) By construction $\psi'$ becomes
injective.

Now determine $s_{\alpha,i}$ by $(*)$ of \ref {3.8} with $m_i:
m\in M_\alpha$ defined by extending $\mathbb Z$ linearly the rule
$(T_{2(n+1-i)})_j=\delta_{i,j}:i,j=1,2,\ldots,n+1$ and similarly
determine $s_{\beta,i}$  by $(*)$ of \ref {3.8} with $m_i: m\in
M_\beta$ defined by extending $\mathbb Z$ linearly the rule
$(T_{2(n-i)+1})_j=\delta_{i,j}:i,j=1,2,\ldots,n$.

\begin {lemma} Set $$\psi'(s_{2i-1})=s_{\alpha,i},\forall i=1,2,\ldots,n+1, \quad
\psi'(s_{2i})=s_{\beta,i}, \forall i=1,2,\ldots,n.$$

Then

(i) $\psi'$ extends to an isomorphism of $W(B_{2n+1})$ onto $W^a$.

Denote this common group by $\widehat{W}$.

(ii) $\psi'$ extends to a $\mathbb Z\widehat{W}$ module
isomorphism of $\mathbb Z\pi$ onto $M_\alpha\alpha +M_\beta\beta$.
\end {lemma}

\begin {proof} For example
$$\alpha^\vee(\alpha_{2i})_i = -(x^2T_{2(n-i)+1})_i =
-(T_{2(n-i)+2}+T_{2(n-i)})_i=-1,$$ which gives
$s_{\alpha,i}\psi'(\alpha_{2i})=
\psi'(\alpha_{2i-1}+\alpha_{2i})$, as required.  Similarly for
example

$$\beta^\vee(\alpha_{2i+1})_i=
-(x^{-1}xT_{2(n-i)})_i=-(T_{2(n-i)+1}+T_{2(n-i)-1})_i=-1,$$ which
gives $s_{\beta,i}\psi'(\alpha_{2i-1})=
\psi'(\alpha_{2i-1}+\alpha_{2i})$, as required.
\end {proof}

\subsection{}\label{10.8}

We have constructed the extended Weyl group $W^a$ from
$W:=<s_\alpha,s_\beta>\cong\mathbb Z_m \ltimes \mathbb Z_2$ when
$m=4n$ and (resp. when $m=4n+2$) and shown it to be isomorphic to
$W(B_{2n})$ (resp. $W(B_{2n+1}$).  However the construction is
rather ad hoc and it is not so obvious what one should do for an
arbitrary finite reflection group.  Form this construction we may
go on to describe the crystals which result in a manner analogous
to the case described in \ref {7.7}.   However this does not seem
to be particularly interesting.   Yet we note that there are now
two ways of interpreting the $B(\infty)$ crystal for $m=6$, either
as a in type $G_2$ or type $D_3$.  However these cannot give the
same crystal as the number of positive roots is different in the
two cases.   This is ultimately a consequence of the failure of
the injectivity of $\psi'$, which we only rather artificially
restored. Again we note the image of the root diagram of $B_3$
under $\psi'$ is just the root diagram of $G_2$ where for example
$\alpha_1$ and $\alpha_3$ coalesce to a single root.

In Figure 12 we have drawn the images under $\psi'$ of the root
diagrams in types $B_4$ given a weight triangularization.  We
recall that in this case $\psi'$ is injective.

\newpage

\bigskip
\begin{figure}
\centering
 \centerline{\epsfysize=0.3\vsize\epsffile{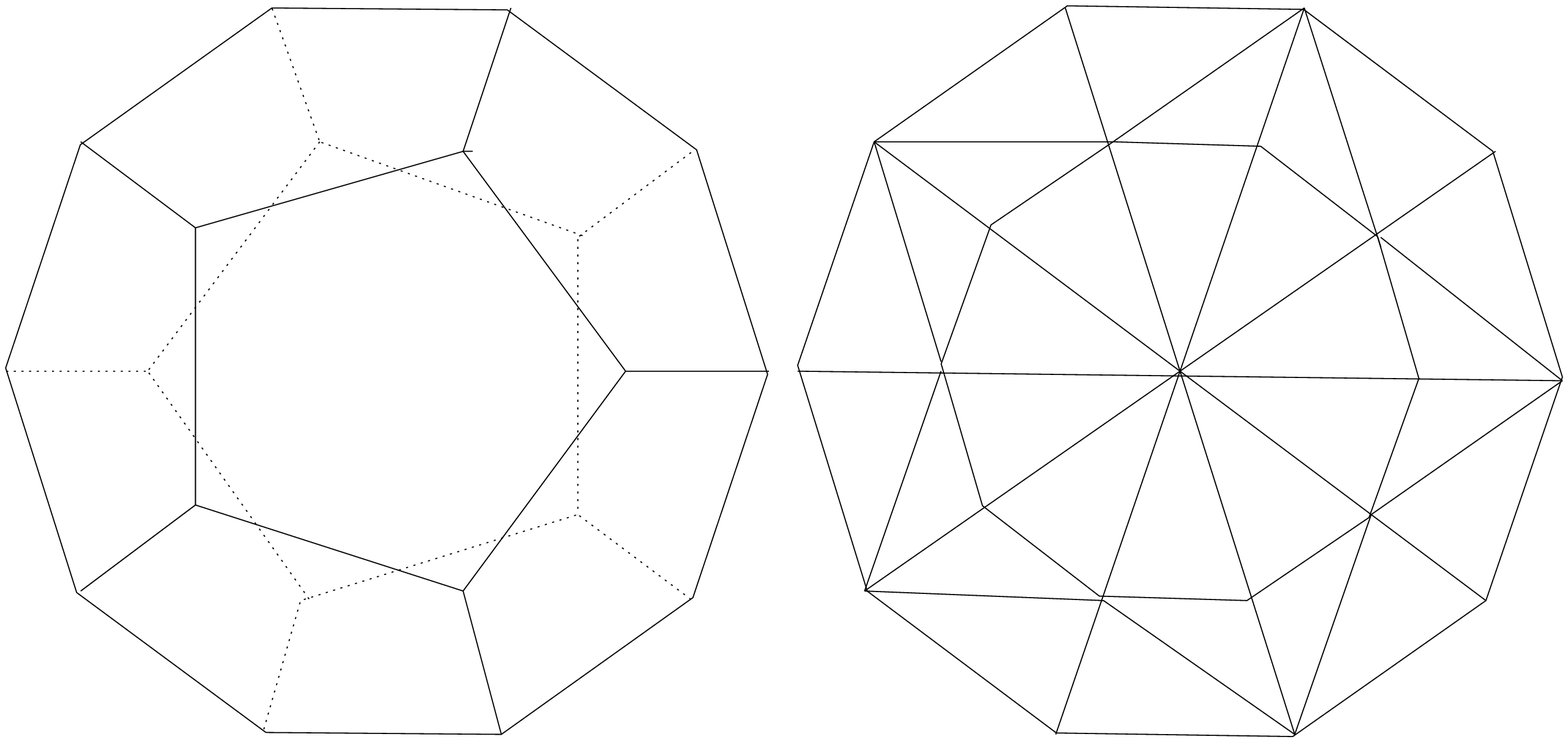}}
\caption{}
 \end{figure}

On the left the dodecahedron with co-ordinates given in \ref
{6.6}.  Projected onto one face it gives the pentagonal root
system for which a crystal in the sense of Kashiwara is
constructed.  On the right the root diagram of $A_4$ presented on
the plane through the map defined in \ref {6.5} and with a weight
triangularization exhibiting a tiling by the Golden Pair.

\newpage

\bigskip
\begin{figure}
\centering
 \centerline{\epsfysize=0.2\vsize\epsffile{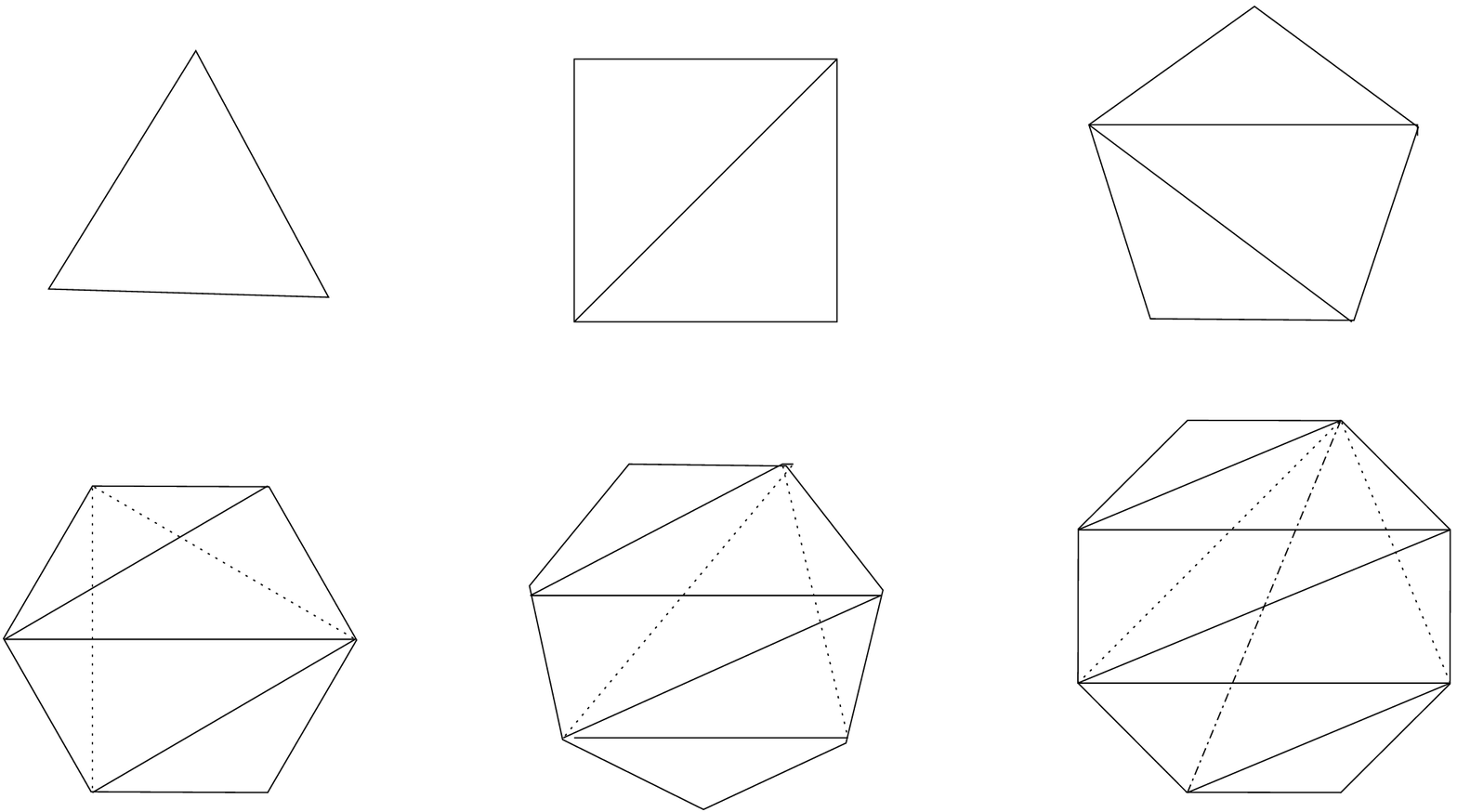}}
\caption{}
 \end{figure}

 Zig-zag triangularization of regular $n$-gons.  For $n=3,4,6$, alcoves
 are obtained (see \ref {9.2}).  For $n=5$ one obtains the Golden Pair
 \ref {8.7}.  For $n \geq 6$ additional triangles may be obtained
 as indicated by the dotted lines. These may be required for
 further weight triangularizations.  For example see Figure 3.

\newpage

\bigskip
\begin{figure}
\centering
 \centerline{\epsfysize=0.35\vsize\epsffile{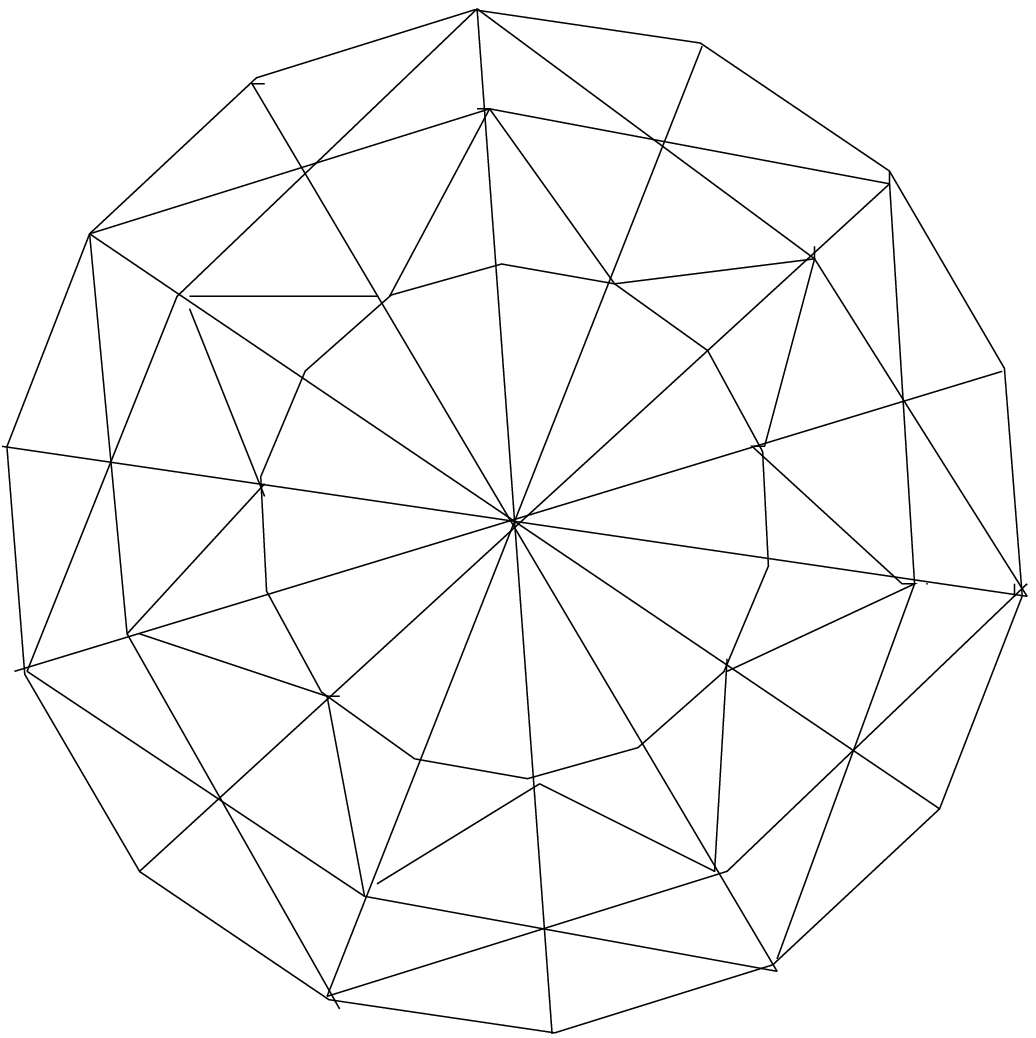}}
\caption{}
 \end{figure}

 The root diagram given a weight triangularization in type $A_6$ presented on the plane
 through the map $\psi'$ defined in \ref {8.3}.

\newpage

\bigskip
\begin{figure}
\centering
 \centerline{\epsfysize=0.35\vsize\epsffile{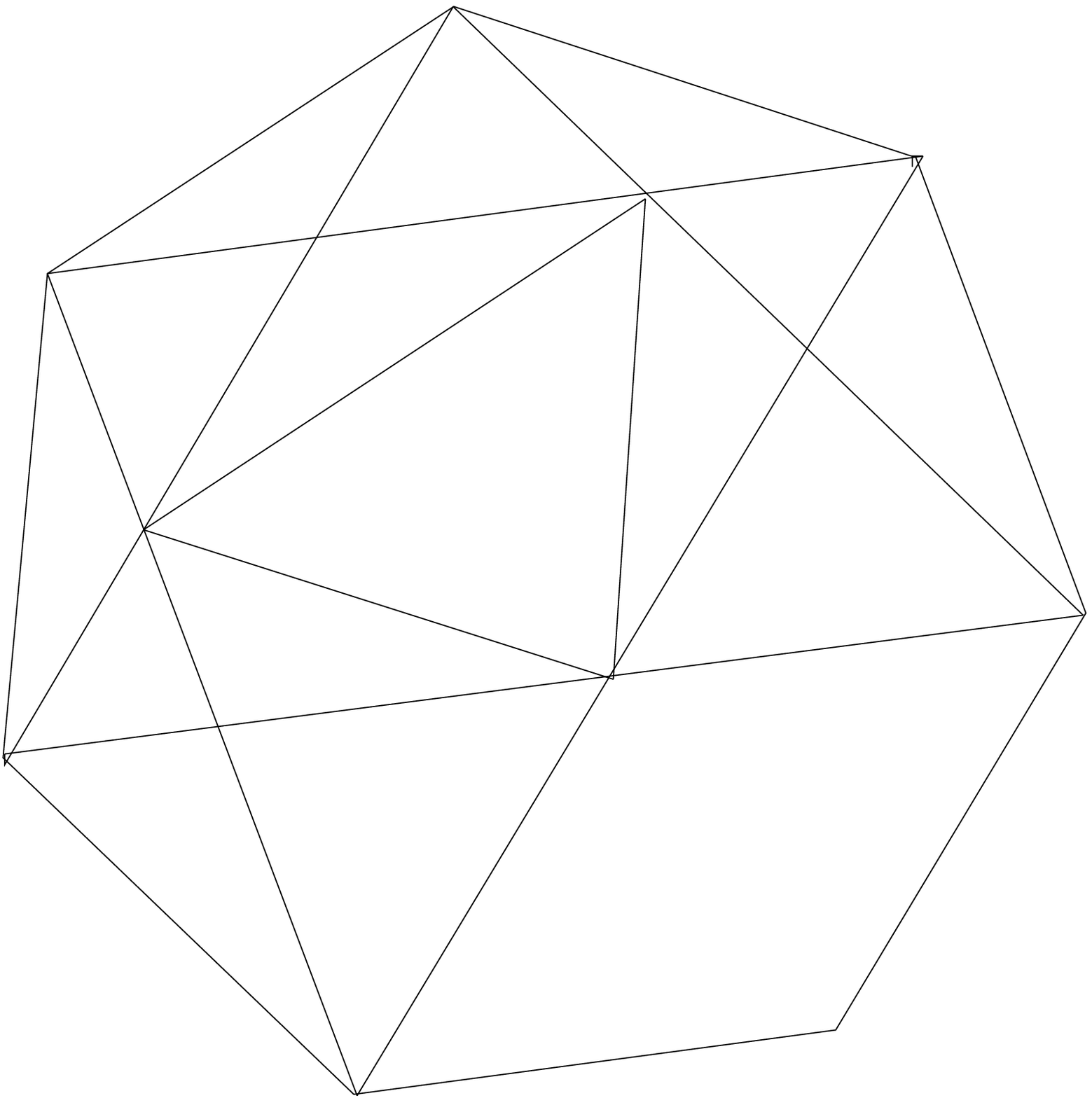}}
\caption{}
 \end{figure}

 The relation $T_0*T_0*T_2*T_3 =p_1T_0$.  Either one of the
 triangles of type $T_0$ with vertices on the regular heptagon is
 cut into four triangles through its intersection with the second
 such triangle.  The resulting four triangles are given by the left hand side
 above.

 \newpage
\begin{figure}
\setlength{\epsfxsize}{1cm}
 \centering

\input Figure5.pstex_t
\caption{}
\end{figure}

\newpage
\newpage

 Symbolic presentation of the computation
 $$p_{2t-1}T_{0,i,j}=p_{t-1}T_{0,i,j}*p_{t-1}T_{0,i+t,j+t}*p_{t-1}T_{0,j-i+t,j}*p_{t-1}T_{0,i,j-t}.$$
 Angle sizes are given up to multiples of $\pi/(2n+1)$,
 having being computed through the indices of the vertices on the circumference.
 Consider the isosceles triangle
 $T:=T\{t,t,2n+1-2t\}$ in the lower left hand corner. Its dotted edge
 has length $s_1=p_{j-i-t-1}$ because it joins the vertices
 $i+1,j$.  Through $T$, this forces
 $s_2=p_{t-1}p_{j-t-i-1}/p_{2t-1}$.  In a similar fashion one
 shows that $s_3 =p_{t-1}p_{j-t-1}/p_{2t-1}$. The triangle $T'$
 with these two edge lengths subtending an angle $i$ is hence
 completed determined and is given by $T'=T\{i,j-i-t,2n+1-j+t\}$.
 This in turn implies that $s_{1,3}=p_{i-1}p_{t-1}/p_{2t-1}$.
 Repeating this computation for the other two sides of the central
 triangle $T''$ shows it to be $p_{t-1}/p_{2t-1}T_{0,i,j}$.  The
 data for the three remaining triangles which form $T_{0,i,j}$ are
 simultaneously obtained and together give the required assertion.

\newpage
\begin{figure}
\setlength{\epsfxsize}{1cm}
 \centering
\input Figure6.pstex_t
\caption{}
\end{figure}

Decomposition of $p_2T\{3,3,3\}$ into nine triangles illustrating
$(*)$ of \ref {8.10}.

\newpage
\begin{figure}
\setlength{\epsfxsize}{1cm}
 \centering
\input Figure7.pstex_t
\caption{}
\end{figure}

Decomposition of $p_2T\{3,3,5\}$ into nine triangles.  To be
contrasted with the previous figure.

\newpage
\begin{figure}
\setlength{\epsfxsize}{1cm}
 \centering
\input fig8.pstex_t
\caption{}
\end{figure}

  Symbolic presentation of the decomposition of
  $p_2T\{i,j,k\}$ into nine triangles in $\mathscr T_{2n+1}$,
  whose angles (as multiples of $\pi/(2n+1)$) and edge lengths are as indicated.
  In this $i,j,k \geq 3$ and $i+j+k = 2n+1$, with $n$ an integer
  $\geq 4$.

\newpage
\begin{figure}
\setlength{\epsfxsize}{1cm}
 \centering
\input fig9.pstex_t
\caption{}

\end{figure}

Symbolic presentation of the decomposition of $p_3T\{i,j,k\}$ into
$16$ triangles with angles as multiples of $\pi/2n+1$ indicated.

\newpage

\bigskip
\begin{figure}
\centering
 \centerline{\epsfysize=0.15\vsize\epsffile{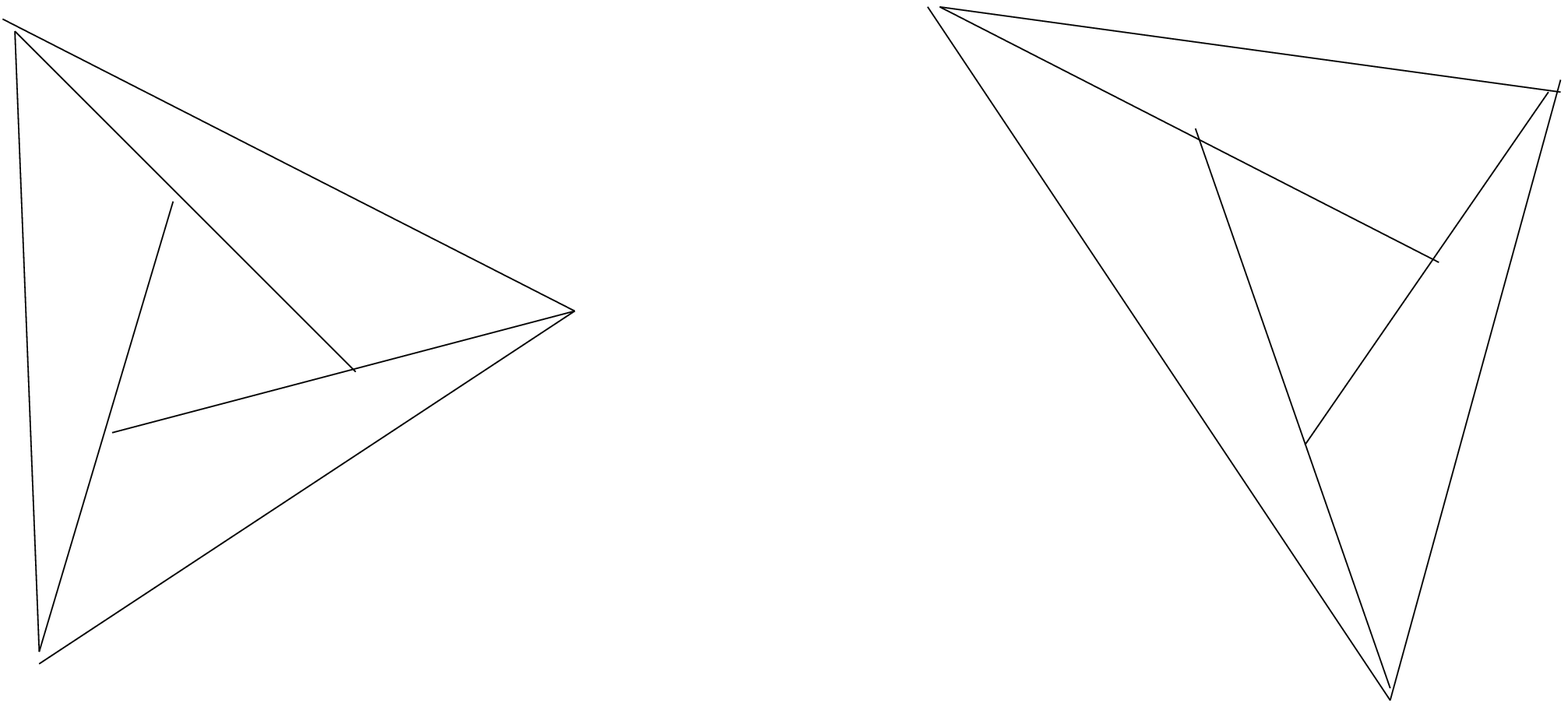}}
\caption{}
 \end{figure}

 Non-standard decompositions of $T\{3,3,3\}$ and $T\{3,3,5\}$,
 that is not satisfying \ref {8.6}.

\newpage
\begin{figure}
\setlength{\epsfxsize}{1cm}
 \centering
\input Figure11.pstex_t
\caption{}
\end{figure}

Detailing the last part of Lemma \ref {9.10}.  One has
$p_i=\varpi_i:i=1,2,3,4$. Then $p_i:i=5,6\ldots,18$ are computed
by vector addition.  One checks that all the latter lie in the
root lattice.  For example
$p_5=\varpi_2+\varpi_3=\alpha_1+2\alpha_2+2\alpha_3+\alpha_4$ and
$p_6=\varpi_1+2\varpi_2=2\alpha_1+3\alpha_2+2\alpha_3+\alpha_4$.

\newpage
\begin{figure}
\setlength{\epsfxsize}{1cm}
 \centering
\input Figure12.pstex_t
\caption{}
\end{figure}

Root diagram in $B_4$ given a weight triangularization.

\

\

\textbf{Acknowledgements}

\

This work was started during a sabbatical spent at the University
of British Columbia and at Berkeley.  I would like to thank my
respective hosts, Jim Carrell and  Vera Serganova for their
invitation. The particular inspiration for this work came from
viewing a presentation of aperiodic Penrose tiling based on the
Golden Pair (see Example 2, \ref {8.7}) displayed in the
mathematics department of the University of Geneva. I should like
to thank Anton Alekseev for his wonderful hospitality and his
enthusiasm for the ideas I presented at his seminar.  I should
like to thank Bruce Westbury for drawing my attention to Chebyshev
polynomials and Anna Melnikov for finding the work interesting, as
well as helping me with Latex. Figures 8, 9 were drawn by our
secretary, Diana Mandelik. I should like to thank her as well
Dimitry Novikov and David Peleg for helping me with the resulting
Latex.

The main results presented here were described in our seminar
"Algebraic Geometry and Representation Theory" at the Weizmann
Institute in August, 2008.

\end{document}